\documentclass[11pt,reqno,oneside]{amsart}
\usepackage{amsmath}
\usepackage{array}
\usepackage{mathabx}
\usepackage{enumerate}
\usepackage{mathrsfs}
\usepackage{dsfont}
\usepackage{amssymb}
\usepackage{graphics}
\usepackage{shuffle}
\usepackage{array}

\usepackage{color}
\newcolumntype{M}[1]{>{\centering\arraybackslash}m{#1}}

\usepackage{scalerel,stackengine}
\stackMath
\newcommand\reallywidehat[1]{%
\savestack{\tmpbox}{\stretchto{%
  \scaleto{%
    \scalerel*[\widthof{\ensuremath{#1}}]{\kern-.6pt\bigwedge\kern-.6pt}%
    {\rule[-\textheight/2]{1ex}{\textheight}}%WIDTH-LIMITED BIG WEDGE
  }{\textheight}% 
}{0.5ex}}%
\stackon[1pt]{#1}{\tmpbox}%
}
\input{amssym.def}
\input{amssym}

\makeatletter
\def\section{\@startsection{section}{1}%
  \z@{1.2\linespacing\@plus\linespacing}{.9\linespacing}%
  {\normalfont\bfseries\large}}
\makeatother

\newtheorem{theorem}{Theorem}

\newtheorem{proposition}{Proposition}
[section]
\newtheorem{lemma}[proposition]{Lemma}
\newtheorem{corollary}[proposition]{Corollary} 

\theoremstyle{definition}

\newtheorem{notation}[proposition]{Notation}

\newtheorem{remark}[proposition]{Remark}

\makeatletter

\@addtoreset{equation}{section}
\makeatother

\newcounter{parage}[section]

\newcounter{parag}[subsection]

\newcounter{paraga}[subsection]

\newcommand{\cL}{{\mathcal L}}
\newcommand{\cB}{\mathcal B}
\newcommand{\B}{{\mathcal B}}

\newcommand{\cP}{{\mathcal P}}
\newcommand{\hcP}{{\widehat\cP}}

\newcommand{\cH}{\mathcal{H}}
\renewcommand{\P}{{\mathcal P}}

\newcommand{\J}{{\mathcal J}}

\newcommand{\vf}{\varphi}

\newcommand{\be}{\begin{equation}}
\newcommand{\ee}{\end{equation}}
\newcommand{\bea}{\begin{eqnarray}}
\newcommand{\eea}{\end{eqnarray}}

\newcommand{\om}{\omega}

\newcommand{\ad}{\operatorname{ad}}

\newcommand{\adham}[1]{{\ad^{\mathrm{Ham}}_{#1}}}

\newcommand{\N}{\mathbb N}

\newcommand{\R}{\mathbb R}
\newcommand{\T}{\mathbb T}
\newcommand{\Z}{\mathbb Z}
\newcommand{\C}{\mathbb C}
\newcommand{\Q}{\mathbb Q}

\newcommand{\kk}{\mathbf{k}}

\newcommand{\de}{\delta}
\newcommand{\lam}{\lambda}
\newcommand{\LA}{\Lambda}
\newcommand{\sig}{\sigma}
\newcommand{\Sig}{\Sigma}
\newcommand{\Th}{\Theta}

\newcommand{\ula}{{\underline\lam}}
\newcommand{\ua}{{\underline a}}
\newcommand{\ub}{{\underline b}}

\newcommand{\ULa}{{\underline\LA}}

\newcommand{\est}{{\scriptstyle\diameter}}
\newcommand{\bul}{\bullet}

\newcommand{\na}{\nabla}

\newcommand{\I}{{\mathrm i}}

\newcommand{\ex}{{\mathrm e}}

\newcommand{\sh}[3]{\operatorname{sh}\!\big( \begin{smallmatrix}#1,\,
#2\\#3\end{smallmatrix} \big)}

\newcommand{\shabl}{\sh{\ua}{\ub}{\ula}}

\newcommand{\ens}{\enspace}
\newcommand{\col}{\colon\thinspace}               %% for a map f \col A \to B
\makeatletter
\newcommand*{\defeq}{\mathrel{\rlap{%
                     \raisebox{0.3ex}{$\m@th\cdot$}}%
                     \raisebox{-0.3ex}{$\m@th\cdot$}}%
                     =}
\makeatother                                                   %%\newcommand{\defeq}{:=}

\newcommand{\ie}{{{i.e.}}\ }

\newcommand{\pa}{\partial}

\newcommand{\tr}[2]{{#1}_{#2}}

\newcommand{\tst}{\textstyle}

\newcommand{\eps}{\varepsilon}

\newcommand{\abs}[1]{\lvert #1 \rvert}
\newcommand{\norm}[1]{\lVert #1 \rVert}

\newcommand{\scal}[2]{{ #1 \cdot #2 }}

\makeatletter
\DeclareRobustCommand{\vf}{\@ifstar\@vf\@@vf}
\newcommand{\@vf}[2]{[ #1, #2 ]_{\mathrm{vf}}}
\newcommand{\@@vf}[2]{\left[ #1, #2 \right]_{\mathrm{vf}}}
\DeclareRobustCommand{\ham}{\@ifstar\@ham\@@ham}
\newcommand{\@ham}[2]{[ #1, #2 ]_{\mathrm{ham}}}
\newcommand{\@@ham}[2]{\left[ #1, #2 \right]_{\mathrm{ham}}}
\DeclareRobustCommand{\qu}{\@ifstar\@qu\@@qu}
\newcommand{\@qu}[2]{[ #1, #2 ]_{\mathrm{qu}}}
\newcommand{\@@qu}[2]{\left[ #1, #2 \right]_{\mathrm{qu}}}
\makeatother

\newcommand{\al}{\alpha}
\newcommand{\ga}{\gamma}

\newcommand{\imp}{\quad \Rightarrow \quad}
\newcommand{\Imp}{\qquad \Rightarrow \qquad}

\newcommand{\dem}{\tfrac{1}{2}}

\newcommand{\IND}[1]{{\mathds 1}_{\{ #1 \}}}

\newcommand{\sst}{\sideset{}{^*}}

\newcommand{\Bula}{B_{[\,\ula\,]}}
\newcommand{\Bul}{B_{[\,\bul\,]}}

\newcommand{\hb}{{\mathchar'26\mkern-9mu h}}
\renewcommand{\hbar}{\hb}
\newcommand{\cE}{{\mathcal E}}

\newcommand{\IM}{\mathop{\Im m}\nolimits}

\newcommand{\rhotau}{\beta}

\newcommand{\bcb }{}
\newcommand{\bcf }{}
\newcommand{\bcbl }{}
\newcommand{\bcg }{}
\newcommand{\ec }{}
\newcommand{\ecf }{}
\newcommand{\ecb }{}
\newcommand{\bl}{}
\newcommand{\ff}{}
\newcommand{\fff}{}
\renewcommand{\widehat}{\reallywidehat}
\begin{document}
%%%%%%%%%%%%%%%%%%%%%%%%%%%%%%%%%%%%%%%%%%%%%%%%%%%%%%%%%%%%%%%%%
%%%%%%%%%%%%%%%%%%%%%%%%%%%%%%%%%%%%%%%%%%%%%%%%%%%%%%%%%%%%%%%%%

\title[]{Normalization in Banach scale Lie algebras via mould calculus and applications}

%%%%%%%%%%%%%%%%%%%%%%%%%%%%%%%%%%%%%%%%%%%%%%%%%%%%%%%%%%%%%%%%%

\author{Thierry Paul}
\address{
CMLS, Ecole polytechnique, CNRS, Universit\'e Paris-Saclay, 91128 Palaiseau Cedex, France
%CNRS and Centre de Math\'ematiques Laurent Schwartz, 
%\'Ecole Polytechnique, 91128 Palaiseau Cedex Fran\-ce
}
\email{thierry.paul@polytechnique.edu}

%%%%%%%%%%%%%%%%%%%%%%%%%%%%%%%%%%%%%%%%%%%%%%%%%%%%%%%%%%%%%%%%%

\author{David Sauzin}
\address{\bl{CNRS UMR 8028 -- IMCCE}\\
\bl{Observatoire de Paris}\\
\bl{77 av. Denfert-Rochereau}\\
\bl{75014 Paris, France}}
\email{david.sauzin@obspm.fr}

%%%%%%%%%%%%%%%%%%%%%%%%%%%%%%%%%%%%%%%%%%%%%%%%%%%%%%%%%%%%%%%%%
%%% PREMIÈRE PAGE
%%%%%%%%%%%%%%%%%%%%%%%%%%%%%%%%%%%%%%%%%%%%%%%%%%%%%%%%%%%%%%%%%

%\setcounter{page}{0}
%\date{}
\date{}
\maketitle

\vspace{2cm}

\begin{abstract}
%
%We establish  perturbative solutions for problems involving
%  resonances of the unperturbed situation, and therefore the necessity
%  of a normal form, in the general framework of Banach scale  Lie algebras defined below. 
  We study a perturbative scheme for normalization problems involving 
resonances of the unperturbed situation, and therefore the necessity of 
a non-trivial normal form, in the general framework of Banach scale Lie 
algebras (this notion is defined in the article).
  This situation covers the case of classical and quantum normal forms in a unified way which allows a direct comparison. In particular we prove a precise estimate for the difference between quantum and classical normal forms, proven to be of order of the square of the Planck constant. Our method uses  mould calculus (recalled in the article) and properties of the solution of  a universal mould equation  
  studied in a preceding paper.
  %\cite{Dyn}.
%
\end{abstract}

%\tableofcontents

%\newpage

\section{Introduction}\label{intro}

\bcg Perturbation theory is a fascinating subject which appears to have been fundamental for the birth of  dynamical systems through Poincar\'e and quantum mechanics in the G\"ottingen school. It is also of fundamental importance for the large computation in physics and chemistry, leading to a panel of different algorithms for computing perturbation series. Each such a method (e.g. generating functions for dynamical systems, functional analysis (expansion of the Neumann series) in quantum mechanics) is very well adapted to emblematic situations (small divisors and KAM theory in classical dynamics, Kato method and existence of dynamics in quantum mechanics), but each methodology seems to be strictly tied to the different underlying paradigms.

%Mould calculus has been introduced and was developed by Jean \'Ecalle
%(\cite{Eca81}, \cite{Eca93}) in the 80-90's in order to
%give powerful tools for handling problems in local dynamics, typically
%the normalization of vector fields or diffeomorphisms at a fixed point.

%Recently mould theory has been used in $\dots$.

%The main goal of 
 In the present article we will present in a unified way  new results
concerning the use of mould theory for (classical) Birkhoff normal
forms (namely in presence of Hamiltonian resonances) and   for quantum
perturbation theory, this last topics having never met, to our
knowledge, mould calculus. 

 As a by-product we also obtain a precise estimate of the difference
 between  quantum normal forms and  the classical ones corresponding
 to the underlying classical  situation, Theorem \ref{corBNF}. Note
 that this estimate is of order of the square of the Planck constant
 and \fff{involves} % involve
only the size of the perturbation.

Mould calculus \fff{was introduced and developed} by Jean \'Ecalle
% has been introduced and was developed by Jean \'Ecalle
(\cite{Eca81}, \cite{Eca93}) in the 80-90's in order to
give powerful tools for handling problems in local dynamics, typically
the normalization of vector fields or diffeomorphisms at a fixed point.

Beside the two topics already mentioned (classical and quantum normal forms), the large
difference of paradigm between them has led us to formulate mould
calculus in a kind of abstract operational setting able to include
both classical  and quantum dynamics, and probably many other situations.

\ec
This formulation leads to mould resolutions of general perturbation
problems, that is problems where a perturbation is added to a bare
problem already explicitly solved.

To put it in a nutshell, one of the key ideas of mould calculus can be
phrased by saying that mould expansions are done on non-universal --
namely related to the perturbation involved in the problem to be
solved -- objects (comould), with universal -- namely dependent only on
the unperturbed, solved problem -- coefficients (mould). This is quite
unfamiliar for people using standard perturbative tools (e.g. Taylor
expansions) where universality is more placed on ``active"
objects. This might explain the poor penetration
%(?) 
of the beautiful theory of moulds in other fields than local dynamics.

Thus, in the present article, we want to consider a general formalism that would include the following cases:
\begin{itemize}
\item the construction of the Birkhoff form for perturbations of
  integrable Hamiltonian systems
  $h(I,\varphi)=h_0(I)+V(I,\varphi), I\in\R^d, \varphi\in\T^d$,
\item the unitary conjugation to a quantum Birkhoff form $\B=
%\B(-\hbar^2\partial_{x_1}^2+\omega_ix_1^2,\dots,-\hbar^2\partial_{x_d}^2+\omega_ix_d^2)
\B(H_1,\dots,H_d)
$ for perturbations of quantum 
``bare" 
%harmonic oscillators
operators $H=
%-\hbar^2\Delta_{\R^d}+\sum\limits_{i=1}^d\omega_ix_i^2+V
H_1+\dots+ H_d,\  [H_k,H_\ell]=0$ 
(e.g. $H_k=-\dem\hbar^2\partial_{x_k}^2+\dem\omega_k^2 x_k^2$ on $L^2(\R,dx_k)$ or 
$H_k=-i\hbar\omega_k\partial_{x_k}$ 
on $L^2(\T,dx_k)$).
\end{itemize}
Let us notice that the following two situations 
have \fff{also} been already considered 
% have been already considered 
via mould theory in our companion  article \cite{Dyn}:
\begin{itemize}
\item
the formal linearization, or at least the formal normalization, of a vector field $X = \sum\limits_{i=1}^N
\om_i z_i \pa_{z_i} + B$ 
(where $B$ represents higher order terms)
in $\C[[z_1,\ldots,z_N]]$,
\item 
the formal symplectic conjugation to a normal form of Hamiltonians
$h(z,\bar z)=\sum\limits_{i=1}^d\dem\omega_i(x_i^2+y_i^2)+V(x,y)$ near the origin.
\end{itemize}$  $
\vskip 1cm

%%%%%%%%%%%%%%%%%%%%%%%%%%%%%%%%%%%%%%%%%%%%%%%%%%%%%%%%%%%%%%%%%
Though these four situations are quite different and belong to different paradigms, we would like to emphasize that mould theory can provide a general formulation handling all of them.

Let us present this general framework. It consists of
\begin{itemize}
\item a Lie algebra $\cL$, 
\fff{which is a Banach scale Lie algebra as defined in
  Section~\ref{defsalg}, or a filtered Lie algebra}
% filtered or a Banach scale Lie algebra defined in Section \ref{defsalg} 
(\fff{the latter case} has been treated in \cite{Dyn}),
% (the case of a filtered algebra has been treated in \cite{Dyn}),

\item assumptions on $\cL$ insuring the existence of an exponential
  map defined on~$\cL$,
\item an element 
$B$ of $\cL$,

\item elements $Y,Z$ of $\cL$ to be determined so that 
\be\label{framework}
[X_0,Z]=0
\ens\text{and}\ens
e^{\ad_Y}(X_0+B)=X_0+Z, 
\ee 
for an ``unperturbed" $X_0 \in \cL$.
\end{itemize}

%Before to present the main results contained in this paper, let us
%mention that the present paper contains the general setting and
%results necessary for the establishment (?) of general explicit
%formulations of perturbative series and convergence results,
%\margeg{annonce un
%peu pr\'ematur\'ee?}
%both topics to appear in two forthcoming article.

Let us present now briefly mould calculus.

Mould theory relies drastically on the notion of \textit{homogeneity},
more precisely on the decomposition of the perturbation into
homogeneous pieces. In the general setting we suppose that the
starting point is an element $X$ of a Lie algebra of the form
\[
X=X_0+B
\]
where $B$ is a ``perturbation" of $X_0$, for which everything is supposed fully known. 

The problem  to solve consists in finding a Lie algebra
automorphism~$\Th$ 
%\margeg{quelles sym\'etries??} \bcg satisfying some symmetries depending on the
%setting\ec 
such that, at any approximation of size any power of $\norm{B}$ for a certain norm $\norm{\cdot}$,
\be\label{prob}
\Th X=\Th(X_0+B)=X_0+Z
\ee
where $Z$ is a normal form, namely a $0-$homogeneous element is a sense we will explain now.

We define an alphabet $\LA \subset \C$ of letters $\lam$
through the decomposition
\[
B=\sum_{\lam\in\LA}B_\lam
\]
where $B_\lam$ satisfies
\be\label{hom}
[X_0,B_\lam] = \lam B_\lam.
%
% [H_0,V_\lam]:=\mbox{ad}_{H_0}V_\lam=\lam V_\lam.
%
\ee
%
%\marge{D: ne pas introduire la notation $\ad$ ici (simple commutateur  pour le moment)}
%
An operator satisfying \eqref{hom} is called $\lam-$homogeneous and $0-$homogeneous operators are called \textit{resonant}.

To the alphabet $\LA$ we can associate the set of \textit{words} 
\be
\ULa \defeq \{\, \ula = \lam_1\lam_2 \cdots\lam_r \mid
r\in\N,\ \lam_i\in\LA \,\}. 
\ee
If $\ula=\lam_1\dots\lam_r$, then we use the notation $r(\ula)
\defeq r$, with the convention $r=0$ for the empty word $\ula=\est$.

We can now define the \emph{Lie comould} as the mapping
\be\label{comould}
\Bul \col \ula \in \ULa \mapsto
\Bula \defeq [B_{\lam_r},[B_{\lam_{r-1}},\ldots [B_{\lam_2},B_{\lam_1}]\ldots]]
\in \cL
\ee
with the convention $B_{[\est]} = 0$ and we call \emph{mould} any mapping
\be\label{mould}
M^\bul \col \ula \in \ULa \in \mapsto
M^{\ula}=M^{\lam_1\cdots\lam_r} \in \C
\ee
\bl{(in this article we use only complex-valued moulds, but \cite{Dyn}
  considers more generally $\kk$-valued moulds, where~$\kk$ is the
  field of scalars of~$\cL$, an arbitrary field of characteristic
  zero).}
To a mould~$M^\bul$, we associate an element of~$\cL$ defined by
\be\label{conca}
M^\bul \Bul \defeq \sum_{\ula\in\ULa} \frac{1}{r(\ula)} M^{\ula} \Bula.
\ee

\smallskip

Returning ot our problem of solving equation~\eqref{prob}, the key
idea will be to process a ``mould ansatz", that is looking to a
solution of \eqref{prob} of the form
\be\label{ansatz}
Z = F^\bul \Bul, \qquad
{ \Th = \sum_{\ula\in\ULa} S^\ula 
\, \ad_{B_{\lam_r}}\cdots \ad_{B_{\lam_1}}, }
\ee
with two moulds $F^\bul$ and~$S^\bul$ to be determined. 

\smallskip

\fff{It turns out} % We will show 
that \eqref{prob} is satisfied through \eqref{ansatz} as
soon as $S^\bullet$ and $F^\bullet$ are solution of the universal
mould equation
\be\label{probmould}
\na S^\bullet=I^\bullet\times {S}^\bullet-{S}^\bullet\times F^\bullet,
\ee
universal because in~\eqref{probmould} the perturbation $B$ does not show up.
 \vskip 0.5cm
In \eqref{probmould} one has
\be\label{derive}
\left\{ \begin{array}{l}
\na M^\bullet \col 
\ula \mapsto \Sig(\ula)M^{\ula} \; \mbox{ where }\Sig(\ula):=\sum\limits_{i=1}^r\lam_i\\[1ex]
M^{\bullet}\times N^{\bullet} \col {\ula}\mapsto
          \sum\limits_{\ula=\ua \, \ub} M^{\ua}N^{\ub}\\[1ex]
I^{\lam_1\dots\lam_r}=\delta_{1r},\mbox{ therefore } I^\bullet\times M^\bullet:\ \ula\mapsto M^{`\ula},\mbox{ with } `\lam_1\dots\lam_r:=\lam_2\dots\lam_r.
\end{array}\right.
\ee
%

%%%%%%%%%%%%%%%%%%%%%%%%%%%%%%%%%%%%%%%%%%%%%%%%%%%%%%%%%%%%%%%%%
Constructing solutions of \eqref{probmould} process in a way familiar
to any perturbative setting: first we note that, precisely because $B$
is perturbation of $X_0$, $\Th$ must be close to the identity and $Z$
to zero. This entails that ${S}^\est=1$ and $F^\est=0$ from which
it follows that
${S}^\bullet\times F^\bullet=F^\bullet+{S}^\bullet\times' F^\bullet$
where
$M^{\bullet}\times' N^{\bullet} \col {\ula}\mapsto
\sum\limits_{\substack{\ula=\ua \, \ub\\r(\underline
    a),r(\ub)<r(\ula)}} M^{\ua}N^{\underline
  b}$.
Moreover willing $Z$ to be $0$-homogeneous is fulfilled by imposing
$F^\bullet$ to be \textit{resonant}, i.e. that
$ \na F^\bullet=0$.  Putting all these properties together
%(see Section \ref{construction} for further details) 
leads to the fact that $F^\bullet$ and the non-resonant part of ${S}^\bullet$ can be determined by induction on the length of letters. What is not determined because it disappears from the equation is the resonant part of ${S}^\bullet$, since it is ``killed" by $\na$.

We  showed in \cite{Dyn} that this ambiguity is removed -- leading to
uniqueness of the solution --  by fixing a \textit{gauge generator}, namely an arbitrary mould $A^\bullet$, resonant 
 %(i.e. $\sig(\ula)=0\Longrightarrow A^{\ula}=0$) 
 and alternal.
 More precisely, for any gauge $A^\bullet$, \eqref{probmould} has a
 unique solution $({S}^\bullet,F^\bullet)$.
 Moreover it happens that
% we will prove that 
 $S^\bullet=e^{G^\bul}$ (where $e$ has to be understood as the exponential in the algebra of moulds, that is $e^{G^\bul}=\sum\limits_{k=0}^\infty \frac{G^{\bul\times k}}{k!}$ where $\times $ is defined in \eqref{derive})
 %is \textit{symmetral} 
 and
 $G^\bul$ and $F^\bullet$ are \textit{alternal}, 
a notion we define now.
 \vskip 0.5cm
 The notion of alternality has to do \fff{with the shuffling} % the notion of shuffling of
 two words $\ua$ and $\ub$, which is the set of words $\ula$ obtained
 by interdigitating the letters of $\ua$ and those of $\ub$ while
 preserving their internal order in $\ua$ or $\ub$. The number of
 different ways a word $\ula$ can be obtained out of $\ua$ and $\ub$
 is denoted by $\shabl$. Saying that $F^\bullet$ is \textit{alternal}
 is nothing but saying that for all non-empty words $\ua,\ub$,\
 $\sum\limits_{\ula\in\ULa} \shabl F^{\ula}=0$.

%%%%%%%%%%%%%%%%%%%%%%%%%%%%%%%%%%%%%%%%%%%%%%%%%%%%%%%%%%%%%%%%%
%%%%%%%%%%%%%%%%%%%%%%%%%%%%%%%%%%%%%%%%%%%%%%%%%%%%%%%%%%%%%%%%%
\vskip 0.5cm
To be more precise, in \cite{Dyn} was proven the following
``existence-uniqueness" result for the mould equation.
\bl{Let us define an operator $\na_1 \col M^\bul \mapsto \na_1 M^\bul$
  by the formula}
\be\label{nabla1} 
\bl{ \na_1 M^\bul \col \ula \in \ULa \mapsto r(\ula)M^{\ula} }
\ee
for an arbitrary mould~$M^\bul$ 
(recall that $r(\ula)$ denotes the length of $\ula$),
and denote by $M^\bul_0$ the resonant part of the mould, defined by
$M^\ula_0 \defeq \IND{\Sig(\ula) = 0} \, M^\ula$ for all $\ula\in\ULa$
(where $\Sig(\ula)$ is by defined in~\eqref{derive}).

%%%%%%%%%%%%%%%%%%%%%%%%%%%%%%%%%%%%%%%%%%%%%%%%%%%%%%%%%%%%%%%%%

\begin{proposition}\label{thmB}
Let $\kk$ be a field of characteristic zero and~$\LA$ a subset of~$\kk$.
For any resonant alternal mould $A^\bul$,
there exists a unique pair $(F^\bul, G^\bul)$ of alternal moulds such
that
\be\label{eqmouldFG}
\na F^\bul = 0, \qquad
\na\big(\ex^{G^\bul}\big) = I^\bul \times \ex^{G^\bul} - \ex^{G^\bul} \times F^\bul,
\ee
\be\label{eqjaugeGA}
\left[ \ex^{-G^\bul}  \times \na_1 \ex^{G^\bul} \right]_0 =A^\bul.
\ee
\end{proposition}

The proof of Proposition~\ref{thmB} is constructive in the sense that
\fff{we obtain} % we will obtain 
the following simple algorithm to compute the values of~$F^\bul$ and
$S^\bul \defeq \ex^{G^\bul}$ on any word~$\ula$ by induction on its
length $r(\ula)$:
introducing an auxiliary alternal mould~$N^\bul$,
one must take $S^\est=1$, $F^\est = N^\est = 0$ and, for $r(\ula)\ge1$,
\begin{align}
\label{eqinducNR}
\Sig(\ula) &\neq0 \imp &
F^\ula &= 0, \quad 
S^\ula = \frac{1}{\Sig(\ula)} \Big( S^{`\ula} -
\sst\sum_{\ula=\ua\,\ub}
S^\ua \, F^\ub \Big), \quad 
N^\ula = r(\ula)\,S^\ula - \sst\sum_{\ula=\ua\,\ub}
S^\ua \, N^\ub, \\[1.5ex]
\label{eqinducRES}
\Sig(\ula)&=0 \imp &
F^\ula &= S^{`\ula} -
\sst\sum_{\ula=\ua\,\ub}
S^\ua \, F^\ub, \quad 
S^\ula = \frac{1}{r(\ula)} \Big( A^\ula + 
\sst\sum_{\ula=\ua\,\ub}
S^\ua \, N^\ub \Big), \quad 
N^\ula = A^\ula,
\end{align}
where we have used the notation
$`\ula \defeq \lam_2\cdots \lam_r$ for $\ula = \lam_1 \lam_2\cdots \lam_r$
and the symbol $\sst\sum$ indicates summation over non-trivial decompositions
(\ie $\ua,\ub\neq\est$ in the above sums);
%
%we will see that 
the mould~$F^\bul$ thus inductively defined is alternal
and
% that
%
\be   \label{eqGlogS}
G^\est = 0, \qquad
G^\ula = \sum_{k=1}^{r(\ula)} 
\frac{(-1)^{k-1}}{k}
\, \sst\sum_{\ula=\ua^1\cdots\ua^k} \,
S^{\ua^1}\cdots S^{\ua^k}
\quad\text{for $\ula\neq\est$}
\ee
then defines the alternal mould~$G^\bul$ which solves~\eqref{eqmouldFG}--\eqref{eqjaugeGA}.

Finally we \fff{proved} % prove 
in \cite{Dyn}, \bl{Propositions~3.8 and~3.9}, the
following result, crucial for the link between the mould equation and
the original problem \eqref{prob}. 
\begin{proposition}\label{prop24}
  If $M^\bul$ and $N^\bul$ are two alternal moulds, then
\[
[M^\bul, N^\bul] \Bul=\bl{[N^\bul \Bul, M^\bul \Bul]},
\]
where $[M^\bul,N^\bul] \defeq M^\bul\times N^\bul - N^\bul\times M^\bul$, and
\[
\ex^{\ad_{M^\bul \Bul}} \big(N^\bul \Bul\big)=\left(\ex^{-M^\bul}\times N^\bul\times \ex^{M^\bul}\right)\Bul.
\]
\bl{Moreover,}
\[
\bl{[X_0, M^\bul \Bul] = (\na M^\bul) \Bul, \qquad
\ex^{\ad_{M^\bul \Bul}} X_0 = X_0 -
\left( \ex^{-M^\bul}\times \na(\ex^{M^\bul}) \right) \Bul.}
\]
\end{proposition}
This result shows that~\eqref{prob} is solved by $Z=F^\bul \Bul$ and $\Th = \ex^{\ad_Y}$
with $Y=G^\bul \Bul$, where~$F^\bul$ and~$G^\bul$
solve~\eqref{eqmouldFG}
\bl{(see \cite{Dyn} for the details)}.
\vskip 1cm

The goal of the present article is twofold: first we want to show how
we can solve perturbatively the normal form problem \eqref{framework}
in the general setting of an $X_0$-extended Banach scale  Lie
algebra -- Theorem \ref{thmD} -- and second we want to show
applications to the aforementioned dynamical problems -- Theorems \ref{thmDclass} and \ref{thmDquant}. As a by-product we give also a quantitative estimate concerning the difference between classical and quantum normal forms -- Theorem \ref{corBNF}.

The different situations in dynamics which can be realized as an 
%
%{$\mathbf{X_0}$\textbf{-extended Banach scale of Lie algebras} 
%
$X_0$-extended Banach scale  Lie algebra
are displayed in the next table.
\scriptsize

 %\vskip 0.5cm
 \begin{center}
   \begin{tabular}{ | r | c | c |c|  }
     \hline
     \begin{tabular}{c}\textbf{{ Banach scale}}\\\textbf{Lie algebra}\end{tabular} &$\begin{array}{l} 
     %Lie\ algebras
     \bcg Banach\ spaces\ec\\(included\ in)  \end{array}$
&$\begin{array}{l}Element\ to\ be\\ normalized\end{array}$& $\begin{array}{l} Normalizing\\transformation\end{array}$   \\ \hline
     
      \hline
  \begin{tabular}{r}near-integrable\\ Hamiltonians\end{tabular} &
 $\begin{array}{c} C^\omega_\rho((\T^n \times \R_+^n)
 %\otimes\C^{2n}
 )= \\
 \mbox{bounded\ functions}\\
\mbox{analytic\ in\ the\ strip}
\\ |\Im z|<\rho \\
\bcg [\cdot,\cdot]=Poisson\ bracket \ec\end{array}$
&
   $\begin{array}{c}
     \mbox{Hamiltonian}\\ H=H_0+V\\
     H_0=\sum\omega_iI_i\\
     V=\sum V_\lam,\\ \lam=\I\,\scal{k}{\om}, \; k\in\Z^n
     \end{array}$&
     $\begin{array}{c} e^{\adham\chi} H = H\circ\Phi \\
     \Phi=\mbox{formal}\\ \mbox{symplectomorphism,}\\ \text{flow of  the v.f.}\;\\ \adham\chi=\{\chi,\cdot\} \\
       
 \end{array}$\\ 
\hline
  \begin{tabular}{r}quantum\\ perturbation \\
  theory 
  \end{tabular}
  &$\begin{array}{c} \{\mbox{pseudodifferential}\\\mbox{operators}\\
  \mbox{of Weyl symbols}\\\mbox{in }C^\omega_\rho|_{\T^n\times\R^n_+}\}\\
  \text{$[\cdot\,,\cdot]_{\mathrm{Q}} = 
  %\frac1{\I\hb}\times
  \frac{\text{commutator}}{i\hb}$}\\ \end{array}$
 &
$\begin{array}{c}
     \mbox{Hamiltonian}\\ H=H_0+V\\
     H_0=\sum
     %\frac12\partial^2_{x_i}+\omega_i^2x_i^2
     E_n|\varphi_n\rangle\langle\varphi_n|\\
     V=\sum V_\lam,\\ \lam=
     %\langle k,\omega\rangle,\;k\in\Z^n
     E_m-E_n
     \end{array}$& 
    $\begin{array}{c}e^{\ad^{Q}_\chi} H=U H U^{-1}\\
     U=e^{i\chi} \\ \text{unitary operator}\\\end{array}$ \\
  \hline
    \begin{tabular}{r}quantum\\ perturbation \\
  theory \\$\hb\to 0$
  \end{tabular}
  &$\begin{array}{c} \{\mbox{pseudodifferential}\\\mbox{operators}\\
  \mbox{of Weyl symbols}\\\mbox{in }C^\omega_\rho|_{\T^n\times\R^n_+}\}\\
  \text{$[\cdot\,,\cdot]_{\mathrm{Q}} = 
  %\frac1{\I\hb}\times
  \frac{\text{commutator}}{i\hb},$}\\
   \hb\to 0 \end{array}$
 &
$\begin{array}{c}
     \mbox{Hamiltonian}\\ H=H_0+V\\
     H_0=
     \\
     \sum\limits_i
     (-\hb^2\partial^2_{x_i}+\omega_i^2x_i^2),
     \\
     %E_n|\varphi_n\rangle\langle\varphi_n|
     V=\sum V_\lam,\\ \lam=
     \scal{k}{\om},\;k\in\Z^n
     %E_m-E_n
     \end{array}$& 
    $\begin{array}{c}e^{\ad^{Q}_\chi} H=U H U^{-1}\\
     U=e^{i\chi} \\ \text{unitary operator}\\
     =\text{quantization of }\Phi\end{array}$ \\
  \hline
   \end{tabular}
 \end{center}
\small 
 \vskip 1cm

The paper is organized as follows.
The first part is devoted
 to the  result valid in any $X_0$-extended Banach scale  Lie
 algebra whose  definition is given in Section \ref{defsalg} and in
 Section \ref{general} we state the general result of the article,
 proven in Section \ref{proofthmD}. 
The second part is devoted to
 applying the main result to classical dynamical situations, Section
 \ref{classintsystquantit}, the quantum ones, Section
 \ref{quantintsysquantit}, and semiclassical approximation,
 Section~\ref{secSemiCl2}. Appendix \ref{allthat} gives the minimal
 setting in semiclassical analysis necessary to the present paper. 
The three other \fff{appendices} % Appendix
provide and prove technical \fff{lemmas} % Lemma
used in  different parts of the article.
 \vskip 0.5cm
Let us finally mention that the present article is self-contained (it
uses only 
%\bcg Proposition \ref{thmD} \ec \margeg{v\'erifier r\'ef. (quelle est
%  cette ``Proposition~A''?)} proven in \cite{PS} and 
%  Proposition
%\bcg 24\ec \margeg{Proposition~24??}
Theorem B
of \cite{Dyn}, rephrased  in Proposition \ref{thmB} of the present article) and all the constants are explicit.

\vspace{1.8cm}
%\newpage

%%%%%%%%%%%%%%%%%%%%%%%%%%%%%%%%%%%%%%%%%%%%%%%%%%%%%%%%%%%%%%%%%
%%%%%%%%%%%%%%%%%%%%%%%%%%%%%%%%%%%%%%%%%%%%%%%%%%%%%%%%%%%%%%%%%

\centerline{\Large\sc Normalization in  $X_0$-extended Banach scale Lie algebras}
%\renewcommand{\theequation}{\thesection.\arabic{equation}}

%\addcontentsline{toc}{part}{\sc Main general results}
\addcontentsline{toc}{part}{\vspace{-1.7em} \\ \strut \hspace{.6em} \sc Normalization in  
$X_0$-extended Banach scale
 Lie algebras\vspace{.4em}}

% \vspace{.3cm}

%%%%%%%%%%%%%%%%%%%%%%%%%%%%%%%%%%%%%%%%%%%%%%%%%%%%%%%%%%%%%%%%%
%%%%%%%%%%%%%%%%%%%%%%%%%%%%%%%%%%%%%%%%%%%%%%%%%%%%%%%%%%%%%%%%%

\section{$X_0$-extended Banach scale Lie algebras}\label{defsalg}

%%%%%%%%%%%%%%%%%%%%%%%%%%%%%%%%%%%%%%%%%%%%%%%%%%%%%%%%%%%%%%%%%
%%%%%%%%%%%%%%%%%%%%%%%%%%%%%%%%%%%%%%%%%%%%%%%%%%%%%%%%%%%%%%%%%

Let $\big(\cL, [\cdot\,,\cdot] \big)$ be a Lie algebra over $\kk = \R$
or~$\C$.
We say that we have an ``$X_0$-extended Banach scale  Lie algebra '' if:
\begin{enumerate}
\item 
  $\cL$ contains a family
  $(\cB_\rho,\norm{\,\cdot\,}_\rho)_{\rho\in\R^*_+}$ of Banach spaces
  over~$\kk$ such that
\[
0<\rho'<\rho \Imp
\cB_{\rho}\subset \cB_{\rho'},
\quad \text{with $\norm{X}_{\rho'} \le \norm{X}_{\rho}$ for all $X\in\cB_{\rho}$,}
\]
\item there exists a constant $\ga>0$ such that
\[
0 < \rho'<\rho''\le\rho, \; X \in \cB_\rho, \; Y \in \cB_{\rho''} 
\Imp
\norm{[X,Y]}_{\rho'} \le \frac{\ga}{e^2(\rho-\rho')(\rho''-\rho')} \norm{X}_{\rho} \norm{Y}_{\rho''},
\]
\item $\cL$ contains an element $X_0$ (which does not necessary belong
  to any of the $\cB_{\rho}$'s) and there exists a function $\chi:
  \R^*_+\to\R^*_+$ such that
\[ 
0 < \rho' < \rho \Imp
\norm{[X_0,Y]}_{\rho'}\leq \frac1{\chi(\rho-\rho')} \norm{Y}_\rho
\ens\text{for all $Y\in\cB_\rho$}. 
\]
\end{enumerate}

Let us denote the adjoint representation of~$\cL$ by $\ad$, \ie for
each $Y\in\cL$, $\ad_Y$ is the Lie algebra derivation defined by
$\ad_Y X = [Y,X]$ for all $X\in\cL$.
One can check (see Corollary~\ref{lemexp}) that, for an $X_0$-extended
Banach scale Lie algebra as above,
if $Y\in\cB_\rho$ satisfies $\norm{Y}_\rho < \rho^2/\ga$,
then $e^{\ad_Y} \defeq \sum\limits_{k\ge0} \frac{1}{k!} (\ad_Y)^k$ 
is a well-defined linear map
\begin{flalign*}
& & e^{\ad_Y} \col \cB_\rho \to \cB_{\rho'}
\quad\text{for each $\rho'$ such that $\tst 0<\rho'<\rho-\sqrt{\ga \norm{Y}_\rho}$} & &
\\
&\text{and}  \hidewidth \\
& & e^{\ad_Y}[X_1,X_2] = 
\big[e^{\ad_Y}X_1, e^{\ad_Y}X_2 \big] 
\quad\text{for all $X_1,X_2\in\cB_\rho$.} & &
\end{flalign*}
Moreover, 
$e^{\ad_Y}X_0 \defeq \sum\limits_{k\ge0} \frac{1}{k!}  (\ad_Y)^k X_0$
too is well-defined and $e^{\ad_Y}X_0 - X_0 \in\cB_{\rho'}$ 
for each $\rho'$ as above.

%%%%%%%%%%%%%%%%%%%%%%%%%%%%%%%%%%%%%%%%%%%%%%%%%%%%%%%%%%%%%%%%%
%%%%%%%%%%%%%%%%%%%%%%%%%%%%%%%%%%%%%%%%%%%%%%%%%%%%%%%%%%%%%%%%%

\section{The general result}\label{general}

%%%%%%%%%%%%%%%%%%%%%%%%%%%%%%%%%%%%%%%%%%%%%%%%%%%%%%%%%%%%%%%%%
%%%%%%%%%%%%%%%%%%%%%%%%%%%%%%%%%%%%%%%%%%%%%%%%%%%%%%%%%%%%%%%%%

\begin{notation}
Let~$\LA$ be a nonempty subset of~$\kk$. 
For a word $\ula=\lam_1\cdots\lam_r\in\ULa$ of length $r\ge1$
and a subset~$\sig$ of~$\{1,\dots,r\}$, we set
\be   \label{eq:defulasig}
\ula_\sig \defeq \sum_{\ell\in\sig} \lam_\ell.
\ee
For $\tau\in\R^*$, we define a function $\rhotau_\tau \col
\ULa \to \R_+$ by the formula
\be\label{deff} 
\rhotau_{\tau}(\ula) \defeq 
\sum_{\substack{\sig\subset\{1,\dots,r(\ula)\}\\ \text{such that}\; \ula_\sig\neq0}}
\, \frac{1}{\abs{\ula_\sig}^{1/\tau}}.
\ee
\end{notation}

%%%%%%%%%%%%%%%%%%%%%%%%%%%%%%%%%%%%%%%%%%%%%%%%%%%%%%%%%%%%%%%%%

\begin{theorem}\label{thmD}
Let~$\cL$ be an $X_0$-extended Banach scale  Lie algebra and let $\rho>0$ and $B\in\cB_\rho$.
Suppose that there exist a subset~$\LA$ of~$\kk$ and a decomposition
\be   \label{eqBdecompBlam}
B = \sum_{\lam\in\LA} B_\lam \qquad
\text{with $B_\lam\in\cB_\rho$ such that
 %and $[X_0,B_\lam] = \lam B_\lam$ for each $\lam\in\LA$.
 }
\ee
%
%Let $F^\bul, G^\bul \in \kk^\ULa$ be defined by~\eqref{eqinducNR}--\eqref{eqGlogS}.
%
%\bcb 
%Let us suppose\ec that $B\in\mathcal B_\rho$ and $B=\sum\limits_{\lambda\in\Lambda}B_\lambda$ for some $B_\lambda\in\mathcal B_\rho$ satisfying

\begin{enumerate}[(i)]
\item $[X_0,B_\lambda]=\lambda B_\lambda$
\item 
%there exists $\tau\geq 1$ and, for all $r\in\N^*$, there exists $\eta_r>0$, such that
%\begin{color}{blue} 
for all $r\in\N^*$, there exist $\eta_r>0$ and $\tau_r\geq 1$, such that
%\[
%\sum\limits_
%%{\lambda_1,\dots\lambda_r\in\Lambda}
%{\underline\lambda=\lambda_1,\dots,\lambda_r\in\underline\Lambda}\norm{B_{\lambda_1}}_\rho\dots
%\norm{B_{\lambda_r}}_\rho 
%e^{\eta_r\sum\limits_{m=1}^{r}\sum\limits_{\underline\lambda_{i_1\dots i_m}\neq 0}|\underline\lambda_{i_1\dots i_m}|^{-1/\tau}}:=\epsilon_r<\infty.
%\] 
%for some $\eta,\mu>0.$
%\[
%\sum\limits_
%%{\lambda_1,\dots\lambda_r\in\Lambda}
%{\underline\lambda=\lambda_1,\dots,\lambda_r\in\underline\Lambda}\norm{B_{\lambda_1}}_\rho\dots
%\norm{B_{\lambda_r}}_\rho 
%e^{\eta_r\sum\limits_{\substack{\sigma\subset\{1,\dots,r\}\\\underline{\lambda}_\sigma\neq 0}}|\underline{\lambda}_\sigma|^{-1/\tau}}:=\epsilon_r<\infty.
%\]
%\[
%\sum\limits_
%%{\lambda_1,\dots\lambda_r\in\Lambda}
%{\underline\lambda=\lambda_1,\dots,\lambda_r\in\underline\Lambda}\norm{B_{\lambda_1}}_\rho\dots
%\norm{B_{\lambda_r}}_\rho 
%e^{\eta_r\sum\limits_{\substack{\sigma\subset\{1,\dots,r\}\\\underline{\lambda}_\sigma\neq 0}}|\underline{\lambda}_\sigma|^{-1/\tau}}:=\epsilon_r<\infty.
%\]
%\[
%\sum\limits_
%%{\lambda_1,\dots\lambda_r\in\Lambda}
%{\underline\lambda=\lambda_1,\dots,\lambda_r\in\underline\Lambda}\norm{B_{\lambda_1}}_\rho\dots
%\norm{B_{\lambda_r}}_\rho 
%e^{\rhotau_{\eta_r,\tau_r}(\underline\lambda) }:=\epsilon_r<\infty.
%\]
\be\label{eqepsrfinite}
\sum\limits_
%{\lambda_1,\dots\lambda_r\in\Lambda}
{\underline\lambda=\lambda_1,\dots,\lambda_r\in\underline\Lambda}\norm{B_{\lambda_1}}_\rho\dots
\norm{B_{\lambda_r}}_\rho 
e^{\eta_r\rhotau_{\tau_r}(\underline\lambda) }:=\epsilon_r<\infty.
\ee
%\end{color}
%\[
%\sum\limits_
%%{\lambda_1,\dots\lambda_r\in\Lambda}
%{\underline\lambda=\lambda_1,\dots,\lambda_r\in\underline\Lambda}\norm{B_{\lambda_1}}_\rho\dots
%\norm{B_{\lambda_r}}_\rho 
%\exp{\big(\underset{\substack{\sigma\subset\{1,\dots,r\},\ \ \ \ \   \\ \underline{\lambda}_\sigma\neq 0\ \ \ \ \ \   }}{\eta_r \sum\ |\underline{\lambda}_\sigma|^{-1/\tau}}\big)}:=\epsilon_r<\infty.
%\]
%\[
%\sum\limits_
%%{\lambda_1,\dots\lambda_r\in\Lambda}
%{\underline\lambda=\lambda_1,\dots,\lambda_r\in\underline\Lambda}\norm{B_{\lambda_1}}_\rho\dots
%\norm{B_{\lambda_r}}_\rho 
%\exp{\big(\underset{\substack{\ \ \ \sigma\subset\{1,\dots,r\}, \\ \ \  \sum\limits_{\ell\in\sigma}\lambda_\ell\neq 0   }}{\eta_r\ \sum\ }|\sum_{\ell\in\sigma}\lambda_\ell|^{-1/\tau}\big)}:=\epsilon_r<\infty.
%\]
\end{enumerate}
%
%We will make two different kinds of hypothesis on the couple $(B, \Lambda)$.
%\vskip 1cm
%\textbf{A}: (i) $\Lambda$ is finite
%
%\hskip 4.5mm(ii) $B=\sum\limits_{\lambda\in\Lambda'}B_\lambda\in\mathcal B_\rho$, for some $B_\lambda\in\mathcal B_\rho$ satisfying $[X_0,B_\lambda]=\lambda B_\lambda$, $\norm{B}_\rho:=\epsilon$
%\vskip 1cm
%\textbf{B}: 
%(i) $B=\sum\limits_{\lambda\in\Lambda}B_\lambda\in\mathcal B_\rho$,  for some $B_\lambda\in\mathcal B_\rho$ satisfying $[X_0,B_\lambda]=\lambda B_\lambda$ and  $\sum\limits_{\lambda\in\Lambda}\norm{B_\lambda}_\rho e^{\eta|\lambda|^{-1}}=\epsilon_\eta=\epsilon<\infty$ for some $\eta>0$,
%
%\hskip 4.5mm(ii) 
%there is a bijection $\lambda:\ \Z^d\to\Lambda$ 
%%$\Lambda$ is 
%such that the two moulds $F^\bullet,G^\bullet$ defined in Theorem \ref{thmA} satisfy, for some 
%%families of positive constants $F_r,G_r,r=1,2,\dots$ and $\tau>\geq 1$
%$\tau\geq1$ and real polynomial $f_r,\ g_r$,
%\begin{eqnarray}
%|F^{\lambda(n_1),\dots,\lambda(n_r)}|&\leq& 
%%F_r|\lambda_1|^{-1}\dots|\lambda_r|^{-1}
%f_r(|n_1|^{\tau},\dots,|n_r|^{\tau})\nonumber\\
%|G^{\lambda(n_1),\dots,\lambda(n_r)}|&\leq& 
%%G_r|\lambda_1|^\tau\dots|\lambda_r|^\tau.
%%g_r(|\lambda_1|^{-1},\dots,|\lambda_r|^{-1})
%g_r(|n_1|^{\tau},\dots,|n_r|^{\tau})
%\nonumber
%\end{eqnarray}
%%where $f_r,\ g_r$ are polynomials.
%% of order $r,\ r-1$ respectively.
%\vskip 1cm
Then,
%, under hypothesis (A) or (B), 
for all $N\in\N^* and\ 0<\rho'<\rho$, 
%there exist $W_N,R_N\in\cB_\rho,\ 
there exists 
%$\Lambda_N=\Lambda_N(\rho')\subset\Lambda$ finite,  
$\epsilon^*=\epsilon^*(N,\rho')$ and $D= D(N,\rho')$, expressed by \eqref{defeps}-\eqref{defD} below, such that, 
%if we define
% for  the same coefficients $F^{\lambda_1,\ldots,\lambda_r},G^{\lambda_1,\ldots,\lambda_r}$ as in Theorem \ref{thmA}, 
 if $F^{\lambda_1,\ldots,\lambda_r},G^{\lambda_1,\ldots,\lambda_r},\ \lambda_1,\dots,\lambda_r\in\Lambda,$ are the  coefficients 
 %defined in Section \ref{intro}
 satisfying \eqref{eqmouldFG}-\eqref{eqjaugeGA} with $A^\bul=0$ and given recursively  by~\eqref{eqinducNR}--\eqref{eqGlogS},
 \begin{enumerate}[(a)]
 \item the two following 
 %series converge (for each $N$) 
 expansions converge in $\mathcal B_{\rho'}$,
\begin{eqnarray}
 \sum_{r= 1}^{N} \,
 \sum_{\lambda_1,\ldots,\lambda_r\in\Lambda}
\frac{1}{r} F^{\lambda_1,\ldots,\lambda_r} 
[B_{\lambda_r},[\ldots[B_{\lambda_2},B_{\lambda_1}]\ldots]]\defeq \ff{Z_N}\in\mathcal B_{\rho'},
\nonumber\\
\sum_{r=1}^{N} \, \sum_{\lambda_1,\ldots,\lambda_r\in\Lambda}
\frac{1}{r} G^{\lambda_1,\ldots,\lambda_r} 
[B_{\lambda_r},[\ldots[B_{\lambda_2},B_{\lambda_1}]\ldots]]\defeq \ff{Y_N}\in\mathcal B_{\rho'}
\nonumber
\end{eqnarray}
\item for $\epsilon_1+\dots+\epsilon_N<\epsilon^*$,

\be\label{resultA3}
\left\{\begin{array}{l}
\ex^{\ad_{\ff{Y_N}}} \Big(X_0+B \Big) =X_0+\ff{Z_N}+\ff{\cE_N},\\ 
\  [X_0,\ff{Z_N}]=0,\\
\norm{\ff{\cE_N}}_{\rho'}\leq D((\epsilon_1+\dots+\epsilon_N)^{N+1}+\epsilon_{N+1}+\dots\epsilon_{N^2})
\end{array}\right.
%\footnote{dans le cas o\`u l' alg\`ebre est r\'ealis\'ee par des op\'erateurs sur un Hilbert, alors, afin d' avoir ici la norme de B et non pas la somme des normes de ses parties homogenes, il suffit remarquer que dans le calcul de la norme de $R_N$, $R_N=R$ est une matrice finie et donc que sa norme est major\'ee par $\sum_i\sup_j|R_{ij}|$. Or les \'el\'ements de matrice des $B_\lambda$ sont ceux de $B$ et donc sont major\'es en valeur absolue par la norme op\'erateur de $B$ elle-m\^eme major\'ee par la norme $\rho$ (voir Porposition 15 de l'artcile avec Laurent. On a donc bien que $\norm{R_N}_{\rho'}\leq C_N\norm{B}\rho$ pour $\norm{B}\rho$ sufisamment petite}.
%%
\ee
\end{enumerate}
(see \eqref{final} for a more precise result).
\end{theorem}

%%%%%%%%%%%%%%%%%%%%%%%%%%%%%%%%%%%%%%%%%%%%%%%%%%%%%%%%%%%%%%%%%

\begin{remark}\label{remark0}
  If in Theorem~\ref{thmD} {we take $B=B(\epsilon)$ depending on a perturbation
  parameter~$\epsilon$ so that $\norm{B_\lam(\epsilon)}_\rho \le
  C_\lam\abs{\epsilon}$ for each $\lam \in \LA$,
with non-negative constants~$C_\lam$,}
then condition~\eqref{eqepsrfinite} factorises and
$\epsilon_r=O(\epsilon^r)$. 
In this case, the final estimates reduces to
$\norm{\cE_N}_{\rho'}=O(\epsilon^{N+1})$. 

Moreover
$\epsilon_1+\dots+\epsilon_N$ and
$\epsilon_{N+1}+\dots+\epsilon_{N^2}$ can be replaced obviously by
${\frac{N(N+1)}{2}} \sup\limits_{r=1\dots N}\epsilon_r$ and
${\frac{N(N^3-1)}{2}} \sup\limits_{r=N+1\dots N^2}\epsilon_r$ respectively
in~\eqref{eqepsrfinite}.
\end{remark}

\begin{remark}\label{remark1}
 Below, in Sections~\ref{classintsystquantit}
  and~\ref{quantintsysquantit}, we will take~$\LA$ of the form
$\LA = \{\, i\,\scal{k}{\om} \mid k\in\Z^d \,\}$
for a given~$\om\in\R^d$.
We shall see that, if there exist $\alpha>0$ and $\tau\geq 1$ such that
the Diophantine condition
  \be\label{diop} \forall k\in\Z^d,\ens \scal{k}{\om} =0
\ens\text{or}\ens
  \abs{\scal{k}{\om}} \geq\alpha|k|^{-\tau} \ee
holds, then one can find $X_0$-extended Banach scale Lie algebras
such that any $B \in \cB_\rho$ has a decomposition
satisfying~\eqref{eqepsrfinite} provided $\tau_r=\tau$ and
  \ff{$\eta_r\le\frac{\rho\,\alpha^{1/\tau}}{2^{r}}$}. 
  Moreover, $\epsilon_r=O(\norm{B}_\rho^r)$ and
  $\norm{\cE_N}_{\rho'}=O(\norm{B}_\rho^{N+1})$ in \eqref{resultA3} in
  this case.
\end{remark}

\begin{remark}\label{remark2}
 There are alphabets for which there exists $C>0$ such that, for each
  $\ula\in\ULa$ and $\sig\subset\{1,\dots,r(\ula)\}$, either
  $\ula_\sig=0$ or $\abs{\ula_\sig}\geq C$. Then,
  condition~\eqref{eqepsrfinite} reduces to
  $\norm{B_\lam}_\rho<\infty$
and entails $\epsilon_r=O(\norm{B}_\rho^r)$ and
  $\norm{\cE_N}_{\rho'}=O(\norm{B}_\rho^r)$. This is the case for
  example in Remark \ref{remark1} in dimension one, or when
  $\omega$ is totally resonant. 
\end{remark}

%%%%%%%%%%%%%%%%%%%%%%%%%%%%%%%%%%%%%%%%%%%%%%%%%%%%%%%%%%%%%%%%%
%%%%%%%%%%%%%%%%%%%%%%%%%%%%%%%%%%%%%%%%%%%%%%%%%%%%%%%%%%%%%%%%%
\section{Proof of Theorem \ref{thmD}}\label{proofthmD}
%%%%%%%%%%%%%%%%%%%%%%%%%%%%%%%%%%%%%%%%%%%%%%%%%%%%%%%%%%%%%%%%%
%%%%%%%%%%%%%%%%%%%%%%%%%%%%%%%%%%%%%%%%%%%%%%%%%%%%%%%%%%%%%%%%%

\subsection{More about the mould equation}\label{mame}

%%%%%%%%%%%%%%%%%%%%%%%%%%%%%%%%%%%%%%%%%%%%%%%%%%%%%%%%%%%%%%%%%

We start by proving the following results concerning
%
%\margeg{Rappeler ce qu'est la mould equation}
%
the solution of the mould equation~\eqref{probmould} as expressed in Proposition \ref{thmB}.

\begin{lemma}
%[Conjecture]
\label{conj}
Let us fix $A^\bullet=0$ in Theorem \ref{thmB}. Then, for the solution
of the mould equation,
$F^{\lam_1\cdots\lam_r}$
(resp. $G^{\lam_1\cdots\lam_r}$) is a linear combination of
inverses of homogeneous monomials of order $r-1$ (resp. $r$) in the
variables
%\lam_i,\lam_{i_1}+\lam_{i_2},\dots,\lam_{i_1}+\lam_{i_2}+\dots+\lam_{i_r} \mid
%\lam_i,\dots,\lam_{i_r}\in\LA
$\big\{ \ula_\sig \mid \sig \subset \{1,\ldots,r\} \big\}$,
with the notation~\eqref{eq:defulasig}: $\ula_\sig = \sum\limits_{\ell\in\sig} \lam_\ell$.

\ff{More precisely,
% setting 
%%
%$\LA_r \defeq \{\, \lam_1+\cdots+\lam_k \mid 1\le k\le r,\;
%\lam_1,\ldots,\lam_k \in \LA \,\}$,
%
% and making use of the notation~\eqref{eq:defulasig}.
%
\begin{align}
\label{eqF}
\ula=\lam_1\dots\lam_r &\Imp
F^{\ula}=\sum_{\substack{\{\sig_j\}_{j=1\dots r-1}\\
\sig_j\subset\{1,\dots,r\}}}
\frac{C^r_{\sigma_1,\dots,\sigma_{r-1}}(\ula)}{\prod\limits_{j=1}^{r-1}\ula_{\sig_j}},
\\[1ex]
\label{eqG}
\ula=\lam_1\dots\lam_r &\Imp
G^{\ula}=\sum_{\substack{\{\sig_j\}_{j=1\dots r}\\
\sig_j\subset\{1,\dots,r\}}}
\frac{D^r_{\sigma_1,\dots,\sigma_{r}}(\ula)}{\prod\limits_{j=1}^{r}\ula_{\sig_j}},
\end{align}
where $C^r_{\sigma_1,\dots,\sigma_{r-1}} \col \underline\Lambda\to\Q$
and $D^r_{\sigma_1,\dots,\sigma_{r}} \col \underline\Lambda \to\Q$
are 
%piecewise constant (with finite number of values) symmetric
bounded
functions such that
 \[
  C^r_{\sigma_1,\dots,\sigma_{r-1}}(\ula)=0\mbox{ when }\prod\limits_{j=1}^{r-1}\ula_{\sigma_j}=0,\quad
  D^r_{\sigma_1,\dots,\sigma_{r-1}}(\ula)=0\mbox{ when }\prod\limits_{j=1}^r\ula_{\sigma_j}=0.
  \]}
\end{lemma}

Note that the sum in \eqref{eqF} (resp.  \eqref{eqG}) contains $\binom{2^{r}}{r-1}$ (resp. $\binom{2^{r}}{r}$) terms.

\ff{\begin{proof}
  The fact of having evaluations of $F^\bullet$, $G^\bullet$ in the
  form of sums of bounded functions divided by monomials in the variables mentioned in
  the statement of Lemma~\ref{conj} is a property obviously stable by
  mould multiplication. Therefore it is enough to prove it for
  $S^\bullet$ in order to get it satisfied for $G^\bullet$. It is
  easily shown to be true by induction using \eqref{eqinducNR} and
  \eqref{eqinducRES} and the fact, easy to prove, that (once again we
  take $A^\bullet=0$)
\[
F^0=1,\  S^0=N^0=0 \mbox{ and }F^\lam=0,\ S^\lam=N^\lam=\frac1\lam,\mbox{ for } \lam\neq0.
\]
 The homogeneity property follows also easily from the induction
generated by \eqref{eqinducNR} and \eqref{eqinducRES}.
\footnote{ We get also the homogeneity by a simple physical
  dimension reasoning: since the letters are defined by
  $\{X_0,B_\lam\}=\lam B_\lam$ and the Poisson bracket by $\{A,B\}=\frac{\partial A}{\partial p}\frac{\partial B}{\partial q}-\frac{\partial A}{\partial q}\frac{\partial B}{\partial p}$, we have that $\lam$ must have the
  dimension of $\frac{energy}{action}$ (the dimension of action is the
  one of $p\times q$). An evaluation of the comould on a word of
  length $r$,
  $\{B_{\lam_r}\{B_{\lam_{r-1}}\{\dots,B_{\lam_1}\}\}\dots\}$ has the
  dimension $\frac{energy^r}{action^{r-1}}$. Finally the dimension of
  the normal form is the one of an energy. Since all the constants in
  the mould equation (with zero gauge) are universal and therefore
  have no dimension, we conclude that the dimension of the evaluation
  of the mould $F^\bullet$ on a word of length $r$ is
  $energy\times
  \frac{action^{r-1}}{energy^r}=\left(\frac{action}{energy}\right)^{r-1}=\left(dimension\
    of\ \lam\right)^{-(r-1)}$.
  In the same way one sees that since one takes the exponential of
  \ff{$\ad_{Y_N}$, $\ad_{Y_N}$} must have no dimension and therefore \ff{$Y_N$} must have
  the dimension of an action and get the desired homogeneity.}.

     Finally  the fact that the functions $C^r_{\sigma_1,\dots,\sigma_{r-1}}$ and $D^r_{\sigma_1,\dots,\sigma_{r}}$ are
bounded comes from the way of solving \eqref{eqinducNR}-\eqref{eqinducRES} by induction on the length of the words and the fact that the possibly unbounded constant (at fixed length $r$ of the word) appearing in \eqref{eqinducNR}-\eqref{eqinducRES} is $\Sigma(\underline\lambda):=\sum\limits_{i=1}^r\lambda_i$ and it appears only in \eqref{eqinducNR} with homogeneity $-1$.
%and the fact that
%the dichotomy which governs the use, for each $\ula$, of one of the
%two equations \eqref{eqinducNR}, \eqref{eqinducRES} is driven by the homogeneous condition  $\sum\limits_{j=1}^r\lambda_j=0$.
%
\end{proof}}

\ff{Since the functions $C^r_{\sigma_1,\dots,\sigma_{r-1}}$ and $D^r_{\sigma_1,\dots,\sigma_{r}}(\ula)$ 
%take only a finite number of values 
are bounded
we can define
\be\label{frgr}
F_r = \sup_{\substack{\ula=\lambda_1\dots\lambda_r\in\underline\Lambda\\
\\
\sig_1,\dots,\sigma_{r-1}\subset\{1,\dots,r\}}} \abs{C^r_{\sigma_1,\dots,\sigma_{r-1}}(\ula)}, \qquad
G_r =\  \ \sup_{\substack{\ula=\lambda_1\dots\lambda_r\in\underline\Lambda\\
\\
\sig_1,\dots,\sigma_r\subset\{1,\dots,r\}}} \abs{D^r_{\sigma_1,\dots,\sigma_{r}}(\ula)}.
\ee
}
%%%%%%%%%%%%%%%%%%%%%%%%%%%%%%%%%%%%%%%%%%%%%%%%%%%%%%%%%%%%%%%%%
\begin{corollary}\label{corconj}
Making use of the notation~\eqref{deff}, we have
\[
|F^{\lam_1,\dots,\lam_r}| \leq F_r
\left(\frac{{\tau_r}}{e\eta_r}\right)^{(r-1)\tau_r}
e^{ \eta_r \rhotau_{\tau_r}(\ula)},
\qquad
|G^{\lam_1,\dots,\lam_r}| \leq G_r
\left(\frac{{\tau_r}}{e\eta_r}\right)^{r\tau_r}
e^{ \eta_r \rhotau_{\tau_r}(\ula)}.
\]
\end{corollary}
%%%%%%%%%%%%%%%%%%%%%%%%%%%%%%%%%%%%%%%%%%%%%%%%%%%%%%%%%%%%%%%%%

\begin{proof}
  We first remark that $F^\bullet$ and $G^\bullet$ are well defined
  for each words, so the denominators in each rational functions
  component don't contain any term of the form
  $\scal{n}{\om}=0$. We finish using first the
  inequality
\be   \label{ineq:magic}
  x < \ff{\Big(\frac{\tau}{e\eta}\Big)^\tau} \, e^{\eta x^{1/\tau}}
\quad \text{for all $\tau,\eta,x>0$}
\ee
  with
  $x= \abs{\ula_\sig}^{-1/\tau_r}$, $\tau=\tau_r$ and $\eta=\eta_r$,
  and second the fact that we have
  $ |C_r(\ula_{\sig_1},\dots,\ula_{\sig_{r-1}}|\leq F_r$
  and  
  $ |D_r(\ula_{\sig_1},\dots,\ula_{\sig_{r}}|\leq G_r$ for all $\sig_1,\dots\sig_r\subset\{1,\dots,r\}$.
\end{proof}

Using \eqref{conclie} \fff{of Lemma \ref{psto} in Appendix~\ref{appLieBr}, we immediately get}
% in Lemma \ref{psto} below we get immediately 
the following result.

\begin{corollary}\label{corcorconj}
Under the hypothesis~\eqref{eqepsrfinite} of Theorem \ref{thmD} we have that
\[
 \norm{\sum_{\lam_1,\ldots,\lam_r\in\LA}
\frac{1}{r} F^{\lam_1,\ldots,\lam_r} |
[B_{\lam_r},[\ldots[B_{\lam_2},B_{\lam_1}]\ldots]]}_{\rho'}
\leq 
\frac{(r-1)!}r \left(\frac\gamma{(\rho-\rho')^2}\right)^{r-1} F_r\left(\frac{{\tau_r}}{e\eta_r}\right)^{{\tau_r}(r-1)}
\epsilon_r
\]
 \[
\norm{\sum_{\lam_1,\ldots,\lam_r\in\LA}
\frac{1}{r} G^{\lam_1,\ldots,\lam_r} 
[B_{\lam_r},[\ldots[B_{\lam_2},B_{\lam_1}]\ldots]]}_{\rho'}
\leq 
\frac{(r-1)!}r \left(\frac\gamma{(\rho-\rho')^2}\right)^{r-1} G_r\left(\frac{{\tau_r}}{e\eta_r}\right)^{{\tau_r} r}
\epsilon_r
\]
\end{corollary}
\vskip 1cm

\subsection{More estimates}\label{me}

The following Lemma is a direct consequence of Corollary \ref{corcorconj} and the definition of $Y_N$ in Theorem \ref{thmD}.

\begin{lemma}\label{RW}
\be\label{conclier}
\norm{Y_N}_{\rho'}\leq 
\sum_{r= 1}^{N}
\frac{(r-1)!}r \left(\frac\gamma{(\rho-\rho')^2}\right)^{r-1}
G_r\left(\frac{{\tau_r}}{e\eta_r}\right)^{{\tau_r} r}\epsilon_r
\fff{\;\, =:\; } % :=
E_{N,\rho-\rho'},
\ee
where $G_r$ is defined in \eqref{frgr}.
\end{lemma}
 \vskip 1cm
 %Let us now p
 \noindent Performing a truncation of $\ex^{\ad_{Y_N}}$ as in Corollary \ref{lemexp}, we get, defining $ \ex^{\ad_{Y_N}}_N=\sum\limits_{d=0}^N\frac1{d!}\underbrace{[Y_N,[Y_N,\dots[Y_N}_{d\ times},\cdot]]]$,
 \begin{eqnarray}\nonumber
\norm{\ex^{\ad_{Y_N}} \Big(X_0+B \Big)-\ex^{\ad_{Y_N}}_N \Big(X_0+B \Big)}_{\rho'}
\nonumber\\
\leq\left(\frac{(\rho''-\rho')^2}{\chi(\rho''-\rho')}+
\norm{B}_{\rho''}\right)\frac{\left(\frac \gamma{(\rho''-\rho')^2}\norm{Y_N}_{\rho''}\right)^{N+1}}{\left(1-\frac \gamma{(\rho''-\rho')^2}\norm{Y_N}_{\rho''}\right)}\nonumber\\
=
\left(\frac{(\rho-\rho')^2}{4\chi(\frac{\rho-\rho'}2)}+
\norm{B}_\rho\right)\frac{\left(\frac {4\gamma}{(\rho-\rho')^2}\norm{Y_N}_{\frac{\rho+\rho'}2}\right)^{N+1}}{\left(1-\frac {4\gamma}{(\rho-\rho')^2}\norm{Y_N}_{\frac{\rho+\rho'}2}\right)}\nonumber
%\\
%=X_0+\widetilde Z_N+O_{\norm{\cdot}_{\rho'}}(\epsilon^{N+1})
%\nonumber\\
%=X_0+Z_N+\cE_N
\end{eqnarray}
by taking $\rho''=\frac{\rho+\rho'}2\leq \rho$ and using $\norm{\cdot}_{\rho''}\leq\norm{\cdot}_\rho$.
\begin{lemma}\label{sergeom}
Let 
\be\label{EN}
E_{N,\frac{\rho-\rho'}2}\leq\frac12 \frac{(\rho-\rho')^2}{4\gamma}.
\ee
Then
\ff{ $\cE^1_N \defeq \ex^{\ad_{Y_N}} (X_0+B )-\ex^{\ad_{Y_N}}_N (X_0+B )$
satisfies}
\begin{eqnarray}\nonumber
\ff{ \norm{\cE^1_N}_{\rho'} }
\leq
C_{N+1}^{sg}(E_{N,\frac{\rho-\rho'}2})^{N+1}\label{sergeom2}
\end{eqnarray}
with
\be\label{csgeom}
C_{N+1}^{sg}=2\left(\frac{(\rho-\rho')^2}{4\chi(\frac{\rho-\rho'}2)}+
\norm{B}_\rho\right)\left(\frac {4\gamma}{(\rho-\rho')^2}
\right)^{N+1}
\ee

\end{lemma}
\vskip 1cm

\subsection{End of the proof}\label{end}

Let us go back now to the mould equation:
\[
\na F^\bul = 0, \qquad
\na\big(\ex^{G^\bul}\big) - I^\bul \times \ex^{G^\bul} \bcf+\ecf \ex^{G^\bul} \times F^\bul=0.
\]
Let us call $F_N$ and $G_N$ the moulds of $Z_N$ and $Y_N$, that is $F_N^{\ula}=F^{\ula}$ if $r(\ula)\leq N$ and $F_N^{\ula}=0$ otherwise.

Let us define $(\ex^{G^\bul})_N$ as for $F_N$ and $G_N$. Obviously
$
(\ex^{G^\bul})_N=(\ex^{G_N^\bul})_N$, and
\[ 
\left(\na\big(\ex^{G^\bul}\big)\right)_N=\left(\na\big(\ex^{G_N^\bul}\big)\right)_N,\  
\left(I^\bul \times \ex^{G^\bul}\right)_N=\left(I^\bul \times \ex^{G_N^\bul}\right)_N, 
%\mbox{ and } 
\left(\ex^{G^\bul} \times F^\bul\right)_N=
\left(\ex^{G_N^\bul} \times F_N^\bul\right)_N.
\]
Moreover  the mould equation reads
\[
%\na F_N^\bul = 0, \qquad
%
\left(\na\big(\ex^{G^\bul}\big) - I^\bul \times \ex^{G^\bul} \bcf+\ecf \ex^{G^\bul} \times F^\bul\right)_N=0,\ \forall N
%\left(\na\big(\ex^{G_N^\bul}\big) - I^\bul \times \ex^{G_N^\bul} - \ex^{G_N^\bul} \times F_N^\bul\right)_N.
\]
 and therefore
\be\label{gnagna}
\left(\na\big(\ex^{G^\bul_N}\big) - I^\bul \times \ex^{G^\bul_N} \bcf+\ecf \ex^{G^\bul_N} \times F^\bul_N\right)_N=0,\ \forall N
\ee

We get 

\begin{eqnarray}\label{finally}
\ex^{\ad_{Y_N}} \Big(X_0+B \Big)\label{fmmm1}\\
=\ex_N^{\ad_{Y_N}}(X_0+B)+\cE^1_N\label{fmm1}\\
=X_0\bcf+I^\bul\Bul\ecf+\sum_{d=1}^N\frac{\ff{(-1)^d}}{d!}\underbrace{[{G^\bullet_N},[{G^\bullet_N},\dots,[{G^\bul_N}}_{d-1\ times},\na {}{G^\bullet_N}+[{}{G^\bullet_N},I^\bul]]]
\Bul+\cE^1_N\label{finallymoins1}\\
=X_0\bcf+I^\bul\Bul\ecf+\left(\sum_{d=1}^N\frac{\ff{(-1)^d}}{d!}\underbrace{[{}{G^\bullet_N},[{}{G^\bullet_N},\dots,[{}{G^\bul_N}}_{d-1\ times},\na {}{G^\bullet_N}+[{}{G^\bullet_N},I^\bul]]]\right)_N
\Bul+\cE^2_N+\cE^1_N\label{finally}\\
=X_0+\left(\bcf-\ecf e^{-G^\bullet_N}\times(\na e^{G^\bullet_N})+e^{-G^\bullet_N}\times I^\bul \times e^{G^\bullet_N}\right)_N
\Bul+\cE^2_N+\cE^1_N\label{finallyplus1}\\
=X_0+F^\bul_N\Bul+\cE^2_N+\cE^1_N\label{fpp1}\\
=X_0+Z_N+\cE_N\nonumber
\end{eqnarray}
with $\norm{\cE^1_N}_{\rho'}\leq C^{sg}_{N+1}E_{N,(\rho-\rho')/2}^{N+1}$ and
\begin{eqnarray}
\norm{\cE^2_N}_{\rho'}\leq2\sum_{r=N+1}^{N^2}
\left(\sum_{d=1}^N\sum_{\substack{k_1+k_2+\dots k_N=d\\
k_1+2k_2+\dots+Nk_N=r}}\frac{N!}{\prod\limits_{i=1}^Nk_i!}\right)%\frac{C^re^{2r}}{2\pi(\rho-\rho')^{2r}}
\frac{(r-1)!}r \left(\frac\gamma{(\rho-\rho')^2}\right)^{r-1} G_r\left(\frac{{\tau_r}}{e\eta_r}\right)^{{\tau_r} r}
\epsilon_r\nonumber\\
\leq2\sum_{r=N+1}^{N^2}
\left(\sum_{d=1}^NN^d\right)%\frac{C^re^{2r}}{2\pi(\rho-\rho')^{2r}}
\frac{(r-1)!}r \left(\frac\gamma{(\rho-\rho')^2}\right)^{r-1} G_r\left(\frac{{\tau_r}}{e\eta_r}\right)^{{\tau_r} r}
\epsilon_r\nonumber\\
\leq2 N^N(E_{N^2,\rho-\rho'}-E_{N,\rho-\rho'})\ \ \  \label{fppp1}
\end{eqnarray}
where $E_{N,\rho}$ is defined by \eqref{EN}.
%
%%%%%%%%%%%%%%%%%%%%%%%%%%%%%%%%%%%%%%%%%%%%%%%%%%%%%%%%%%%%%%%%%%
%%%%%%%%%%%%%%%%%%%%%%%%%%%%%%%%%%%%%%%%%%%%%%%%%%%%%%%%%%%%%%%%%%
%
Let us explain how we derive the chain of inequalities after \eqref{fmm1}:
\begin{itemize}
\item \eqref{fmmm1}$\Rightarrow$\eqref{fmm1}; is Lemma \ref{sergeom}.
\item \eqref{fmm1}$\Rightarrow$\eqref{finallymoins1}; writing the first part of \eqref{fmm1} in mould calculus by Proposition \ref{prop24}.

\item \eqref{finallymoins1}$\Rightarrow$\eqref{finally} with \eqref{fppp1}:  since  the support of $G^\bul_N$ contains only  words of length up to $N$, we can expand the mould in \eqref{finallymoins1} up to words of length $N$, which appears in \eqref{finally}, plus the rest. The rest, whose support contains only words on length between $N+1$ and $N^2$,  gives \eqref{fppp1} by combinatorial coefficients and Corollary \ref{corcorconj}.

\item \eqref{finally}$\Rightarrow$\eqref{finallyplus1}; is obtained
   by \bcf decomposing 
  \[ - e^{-G^\bullet_N}\times(\na
  e^{G^\bullet_N})+e^{-G^\bullet_N}\times I^\bul \times
  e^{G^\bullet_N}
= I^\bul + 
\sum\limits_{d=0}^\infty\frac{(-1)^d}{d!}
  \underbrace{[{G^\bullet_N},[{G^\bullet_N},\dots,[{G^\bul_N}}_{d-1\
    times},\na {}{G^\bullet_N}+[{}{G^\bullet_N},I^\bul]]]
  \] (see also Propositions~3.8(ii) and~3.9(ii) of \cite{Dyn})\ecf
  and noticing that, since $G^\emptyset_N=0$, one has that
  $\underbrace{[{G^\bullet_N},[{G^\bullet_N},\dots,[{G^\bul_N}}_{M\
    times,\ M>N},\na
  {}{G^\bullet_N}+[{}{G^\bullet_N},I^\bul]]] $
  contains only word on length greater than $N$.

\item \eqref{finallyplus1}$\Rightarrow$\eqref{fpp1}; is obtained  by \eqref{gnagna}.

\item \eqref{fpp1}$\Rightarrow$\eqref{fppp1}; by writing
  $\cE_N:=\cE^1_N+\cE^2_N$.
 \end{itemize}
 \vskip 1cm

Therefore 
\be\label{final}
\norm{\cE_N}_{\rho'}\leq
C^{sg}_{N+1}E_{N,(\rho-\rho')/2}^{N+1}+N^N(E_{N^2,\rho-\rho'}-E_{N+1,\rho-\rho'})\mbox{
  for } E_{N,\frac{\rho-\rho'}2}\leq\frac12
\frac{(\rho-\rho')^2}{4\gamma}
\ee 
where $C^{sg}_{N+1}$ and $E_{N,(\rho-\rho')}$ are defined in
\eqref{csgeom} and \eqref{conclier}. We now take
\be\label{defeps}\epsilon^*=
\frac{(\rho-\rho')^2}{32\gamma\sup\limits_{r=1\dots N}\frac{(r-1)!}r
  \left(\frac\gamma{(\rho-\rho')^2}\right)^{r-1}\left(\frac{2^r{\tau_r}}{e\eta_r}\right)^{r}}
=\frac{(\rho-\rho')^2}{32\gamma
\frac{(N-1)!}N \left(\frac\gamma{(\rho-\rho')^2}\right)^{N-1}
\left(\frac{2^N{\tau_r}}{e\inf\limits_{r=1\dots N}\eta_r}\right)^{N}
}
\ee
 and 
$$D=C^{sg}_{N+1}\left(4
\frac{(N-1)!}N \left(\frac\gamma{(\rho-\rho')^2}\right)^{N-1}
\left(\frac{2^N{\tau_r}}{e\inf\limits_{r=1\dots N}\eta_r}\right)^{N}
\right)^{N+1}$$
\be\label{defD}
+N^N
\frac{(N^2-1)!}{N^2} \left(\frac\gamma{(\rho-\rho')^2}\right)^{N^2-1}
\left(\frac{2^{N^2}{\tau_r}}{e\inf\limits_{r=N+1\dots N^2}\eta_r}\right)^{N^2}.
\ee 
Theorem \ref{thmD} is proved.

%\vspace{1.8cm}
%\newpage
\vskip 1cm
%%%%%%%%%%%%%%%%%%%%%%%%%%%%%%%%%%%%%%%%%%%%%%%%%%%%%%%%%%%%%%%%%
%%%%%%%%%%%%%%%%%%%%%%%%%%%%%%%%%%%%%%%%%%%%%%%%%%%%%%%%%%%%%%%%%

\centerline{\Large\sc Applications to dynamics}
%\addcontentsline{toc}{part}{\sc Main general results}
%\begin{color}{red}
Normal forms have a long history since the seminal work by Poincar\'e in perturbation theory \cite{P}. See also \cite{Bi} for a more ``dynamical systems" presentation.  Their use in stability problems for dynamical systems are presented in the  textbooks \cite{MS,A,G,L}. More recent results and surveys are present in the articles \cite{PM,Z,LS7}.

Bohr-Sommerfeld quantization of normal forms have been used before the birth of quantum mechanics itself (namely the publication of \cite{H}): see \cite{B}. More recently quantum normal forms have been used in spectral problems near minima of potentials (\cite{S}, \cite{BGP} and \cite{CV}, for inverse problems in, e.g.,  \cite{ISZ} and \cite{GP3} and perturbations of integrable systems in \cite{GP,PS}. 

The link between quantum and classical normal forms \fff{has} % have 
been established in \cite{GP1} and \cite{DGH}.
\vskip 1cm
The three Theorems \ref{thmDclass}, \ref{thmDquant} and \ref{corBNF} below give systematic precise 
estimates for the construction of normal forms at any order as rephrasing of the general Theorem \ref{thmD}. The estimate between quantum and classical normal forms contained in Theorem \ref{corBNF} is to our knowledge new.
%\end{color}
\addcontentsline{toc}{part}{\vspace{-1.7em} \\ \strut \hspace{.6em} \sc Applications to dynamics\vspace{.4em}}

% \vspace{.3cm}

%%%%%%%%%%%%%%%%%%%%%%%%%%%%%%%%%%%%%%%%%%%%%%%%%%%%%%%%%%%%%%%%%
%%%%%%%%%%%%%%%%%%%%%%%%%%%%%%%%%%%%%%%%%%%%%%%%%%%%%%%%%%%%%%%%%

\section{Quantitative classical formal normal forms}\label{classintsystquantit}

We denote the circle by
\[ \T \defeq \R/2\pi\Z. \]
Let $d\ge 1$ be integer.
We are interested in two situations: the phase-space~$\cP$ is
either \bcb $$T^*\R^d \simeq \R^d \times \R^d\mbox{ and then
}X_0(x,\xi)=\frac12\Big(
\sum\limits_{j=1}^d\xi_j^2+\sum\limits_{j=1}^d\omega_j^2x_j^2 \Big),$$ 
or it is $$T^*\T^d \bl{\simeq \T^d \times \R^d}\mbox{ and  }X_0(x,\xi)=\omega\cdot\xi.$$ 
In both cases, we suppose that $\omega\in\R^d$ {has components
$\om_j>0$}, and we will denote the variable in~$\cP$ by
$(x,\xi)$.
\bcb The symplectic $2$-form being $\sum d\xi_j \wedge d x_j$, the
Hamiltonian vector field associated with~$X_0$ is
$\{X_0,\cdot\} = \sum \big( \xi_j \frac{\pa\,\;}{\pa x_j} 
- \om_j^2 x_j \frac{\pa\,\;}{\pa \xi_j}  \big)$
in the first case, and $\sum \om_j \frac{\pa\,\;}{\pa x_j}$ in the second.
%
%We take for~$\cL$ the space of all real-analytic functions on~$\cP$,
%with the Lie bracket given by the Poisson bracket.
%
\ec

We want to perturb $X_0$ by a ``small" perturbation $B$ and want to
show that it is possible, after a symplectic change of coordinates, to
put the new Hamiltonian $X_0+B$ into a \textit{normal form}
$X_0+Z,\ \{X_0,Z\}=0$, modulo an {arbitrarily small} error. 
%
%we want to be as small as we wish. 
%
%Here $\{\cdot,\cdot\}$ denotes the Poisson bracket on $\cP$.

The result will be expressed in Theorem \ref{thmDclass} below, which
will follow from Theorem \ref{thmD}. We first have to show how our
situation enters in the framework of the first part of this article.

Let 
\[
\hcP \defeq \R^d \times \R^d
\quad\text{if $\cP = \R^d\times\R^d$,} \qquad
\hcP \defeq \R^d \times \Z^d 
\quad\text{if $\cP=T^*\T^d$.} 
\]
We define the (symplectic) Fourier transform $\widehat G$ of a
function $G \in L^1(\cP,dxd\xi)$ by
\be\label{fou}
\widehat G(q,p)=\frac1{(2\pi)^{d}} \int_\cP G(x,\xi)e^{-i({px-q\xi})}dx d\xi
\quad \text{{for $(q,p)\in\hcP$.}}
\ee
%
%\margeg{J'ai chang\'e le signe dans l'exp pour la transfo inverse. OK?}
%
Let $d\mu$ denote either the Lebesgue measure $dqdp$ on $\R^d\times\R^d$
or the product of the Lebesgue measure $dq$ by
the sum of Dirac \ff{masses} % mass
on $\Z^{d}\subset\R^d$ (\ff{counting measure}).
% (counting functions).
%

If $\widehat G \in L^1(\hcP,d\mu)$, then
\be\label{invfou}
G(x,\xi)=\frac1{(2\pi)^{d}}\int \widehat G(q,p)e^{i({px-q\xi})}d\mu(q,p)
\quad \text{for a.e.\ $(x,\xi) \in \cP$.}
\ee
\textbf{By a slight abuse of notation, from now on, we will denote $d\mu(q,p)$ by $dqdp$ in both cases.}\label{dpdq}

%%%%%%%%%%%%%%%%%%%%%%%%%%%%%%%%%%%%%%%%%%%%%%%%%%%%%%%%%%%%%%%%%

\bcb
%
% In both cases the ``torus Hamiltonian flow" (torus action)
% $$\Phi_0^t: t\in\T^d_\omega:=\R^d/\frac{2\pi}\omega\Z^d\mapsto 
% e^{t\cdot Jd{X_0}}$$ 
% (where $JdX_0$ is the Hamiltonian vector field associated to $X_0$)
% is globally defined on $\cP$ and we can therefore define
%
Let us write
$X_0 = X_{0,1} + \cdots + X_{0,d}$ with, for each $j=1,\ldots,d$,
\[
X_{0,j} \defeq \dem( \xi_j^2 + \om_j^2 x_j^2 ) 
\ens\text{on $\cP=T^*\R^d$,}
\qquad
X_{0,j} \defeq \om_j \xi_j
\ens\text{on $\cP=T^*\T^d$.}
\]
Since the $X_{0,j}$'s Poisson-commute and since, for each~$j$, all the
solutions of the Hamiltonian vector field $\{X_{0,j},\cdot\}$ are
$\frac{2\pi}{\om_j}$-periodic, we get an action of~$\T^d$ on~$\cP$ by
defining
\[
\Phi_0^t \defeq \exp\bigg( \bigg\{ 
\frac{t_1}{\om_1} X_{0,1} + \cdots + \frac{t_d}{\om_d} X_{0,d},
\,\cdot\, \bigg\} \bigg)
\quad\text{for $t=(t_1,\ldots,t_d) \in \T^d$.}
\]
Given $k\in\Z^d$ and an integrable function~$G$, we now define
\ec
\[
G_{(k)}(x,\xi) \defeq
\bl{ \frac{1}{(2\pi)^d} } \int_{\bl{\T^d}} G(\Phi_0^t(x,\xi))e^{-ikt}dt
\quad \text{{for $(x,\xi)\in\cP$.}}
\]
\bcb
%
% Since, by Poisson formula, $\sum\limits_{k\in\Z^d} G_{(k)}=G$ we get that
% %
% \[
% G(x,\xi)=\frac1{(2\pi)^d}\sum_{k\in\Z^d}\int \widehat{G_{(k)}}(q,p)
% { e^{i(px-q\xi)} }dqdp.
% \]
%
\begin{lemma}   \label{lemSumGkhomog}
For any real-analytic~$G$, one has
$\sum\limits_{k\in\Z^d} G_{(k)} = G$ (pointwise convergence on~$\cP$)
and, for each $k\in\Z^d$,
\be   \label{eq:Gkhomog}
\{ X_0, G_{(k)} \} = i \scal{k}{\om} \, G_{(k)}.
\ee
\end{lemma}

\begin{proof}
For each $(x,\xi)\in\cP$, the function $t\in\R^d \mapsto
G\circ\Phi_0^t(x,\xi)$ is analytic and $2\pi$-periodic in each~$t_j$;
for each $k\in\Z^d$, its $k$th Fourier
coefficient is~$G_{(k)}(x,\xi)$.
The first statement thus follows from the fact that
$G\circ\Phi_0^t(x,\xi)$ is the sum of its Fourier series.

For each~$j$,
$\{ X_{0,j}, G_{(k)} \}$ is the $k$th Fourier coefficient of
the function $\{ X_{0,j}, G\circ\Phi_0^t \}$, 
and this function coincides with
$\om_j \frac{\pa\,\;}{\pa t_j} \big( G\circ\Phi_0^t \big)$,
hence $\{ X_{0,j},  G_{(k)} \} = i k_j \om_j G_{(k)}$,
and the second statement follows.
\end{proof}

\ec

\noindent
%Finally let us denote  by $\widehat X_0$ the function defined on 
%$\widehat{\cP}$ defined by
%$
%%
%\widehat X_0(q,p)=X_0(q,p)
%%
%$
%in both cases 
% $\cP = \R^d\times\R^d$ and $\cP=T^*\T^d$. 
%%
%\margeb{Paragraphe \`a revoir?}
%%
%Since $\widehat X_0(q,p)$ depends only on $p$ in the case $\cP=T^*\T^d$ we can define the symplectic torus action
% \[
% \widehat{\Phi}_0^{t}: t\in\T^d_\omega:=\R^d/\frac{2\pi}\omega\Z^d\mapsto 
%e^{-t\cdot Jd{X_0}}
% \]
%and it is easy to check that, defining $(x,\xi)\cdot(p,q):=px-q\xi$, we have
%\be\label{symp0}
%{\Phi}_0^{t}(x,\xi)\cdot\widehat{\Phi}_0^{-t}(q,p)=(x,\xi)\cdot(q,p).
%\ee
%defining now
%\[
%(\widehat F)_{(k)}(q,p):=
%%
%\int_{\T^d_\omega} \widehat F(\Phi_0^t(q,p))e^{-ikt}dt
%\]
%we get by \eqref{symp0} and the conservation of the Liouville measure by symplectomorphisms
%that
%%, since $\Phi^t_0$ is a linear symplectomorphism, we have that
%%
%\be\label{symp}
%\widehat{F_{(k)}}=(\widehat F)_{(k)}.
%\ee
%%

For $\rho>0$ we will denote by~$\mathcal J_\rho$ the space of all
integrable
functions $G$ whose associated family of functions $G_{(k)}$ have a
Fourier transform whose modulus is integrable with respect to $dqdp$
weighted by $e^{\rho(|q|+|p|)}$,
the family of integrals obtained that way being itself summable with
the weight $e^{\rho|k|}$.

 Namely,\footnote{Note that in the case of \bl{$T^*\T^d$}, $G_{(k)}$ is nothing but the
Fourier coefficient of $G(\cdot,\xi)$ times $e^{-ik\cdot}$.
Therefore in this case $\widehat{G_{(k)}}(q,p) = \delta_{k,p} \widehat{G_{(k)}}(q,p)$ and so
$\norm{G}_\rho=\int |\widehat G(q,p)|e^{\rho(|q|+2|p|)}dqdp$. We
present nevertheless the two cases \bl{($T^*\T^d$ and $T^*\R^d$)} in a
unified way.}
we set 
$\mathcal J_\rho \defeq \{\, \bl{\text{}}\; G \in L^1(\cP,dxd\xi) \mid 
\norm{G}_\rho < \infty \,\}$,
with
\[
\norm{G}_\rho \defeq \sum\limits_{k\in\Z^d}\int_{\hcP}
 |\widehat G_{(k)}(q,p)| e^{\rho(|q|+|p|+|k|)} dqdp.
\]

\bcb
Note that each function in $\mathcal J_\rho$ is real analytic and has a bounded
holomorphic extension to the complex strip
$\big\{\, (x,\xi) \in \C^d\times\C^d \mid 
\abs{\IM x_1}, \ldots, \abs{\IM x_d}, \abs{\IM \xi_1}, \ldots, \abs{\IM \xi_d} 
<\rho \,\big\}$.   \ec
$\mathcal J_\rho$ is obviously a Banach space
satisfying $\Vert G\Vert_{\rho'}\leq \Vert G\Vert_{\rho}$ whenever
$\rho'<\rho$.

\bcg
In the case 
 $\cP = \R^d\times\R^d$ let us denote  by $\widehat X_{0,1},\dots,\widehat X_{0,d}$ the functions defined on 
$\widehat{\mathcal P}$  by
$
\widehat X_{0,j}(q,p):=X_{0,j}(p,q)
$ (order of variables \ff{reversed}) % unversed) 
and let $\widehat{\Phi}_0^{t}$ be the corresponding torus action on $\widehat\cP$.
It is easy to check that, defining 
\be\label{defform}
(x,\xi)\cdot(p,q):=px-q\xi,
\ee
 we have
\be\label{symp0}
{\Phi}_0^{t}(x,\xi)\cdot\widehat{\Phi}_0^{t}(q,p)=(x,\xi)\cdot(q,p).
\ee
Defining now
\[
(\widehat F)_{(k)}(q,p):=
\frac1{(2\pi)^d}\int_{\T^d} \widehat F(\widehat\Phi_0^t(q,p))e^{-ikt}dt
\]
we get by \eqref{symp0} and the conservation of the Liouville measure by symplectomorphisms
that
%, since $\Phi^t_0$ is a linear symplectomorphism, we have that
%
\be\label{symp}
\widehat{F_{(k)}}=(\widehat F)_{(k)}.
\ee

In the case 
 $\cP=T^*\T^d$ we get easily
 \be\label{symp2}
\widehat{F_{(k)}}(q,p)=\widehat F(q,p)\delta_{k,p}.
\ee
 \ec

\vskip 0.5cm
\bcg Let $\cL$ be the space of real analytic functions on $\cP$. 
\fff{The two lemmas of Appendix~\ref{appPoisson} show} % We just proved 
that $(\cL,[\cdot,\cdot])$ endowed with $\mathcal J_\rho=\cB_\rho,\ 0<\rho<\infty$ and with $ [\cdot,\cdot]=\{\cdot,\cdot\},\  \gamma=1\mbox{ and }\chi(\rho)=\frac1{e\rho}$, is an $X_0$-extended Banach scale  Lie algebra.\ec

\bcb Let us remark now that the homogeneous components of a perturbation~$B$ 
are easily deduced from the family $(B_{(k)})_{k\in\Z^d}$.
Indeed, let us define
\be   \label{eqdefLAom}
\LA = \{\, i\,\scal{k}{\om} \mid k\in\Z^d \,\}
\ee
as in Remark~\ref{remark1}. In view of Lemma~\ref{lemSumGkhomog},
we see that, for each $\lam\in\LA$, 
\be\label{eq:decompBlam}
\{X_0,B_\lam\} = \lam B_\lam
\quad\text{with}\ens
B_\lam \defeq \sum_{ k\in\Z^d \mid i\scal{k}{\om} = \lam} B_{(k)},
\ee
and $B = \sum\limits_{\lam\in\LA} B_\lam$.
\ec
Moreover, since
\[
(F_{(k)})_{(k')}= \bl{ \frac{1}{(2\pi)^{2d}} }\int
F\circ\Phi^t_0\circ\Phi^{t'}_0e^{\bl{-i(kt+k't')}}dtdt'
= \bl{ \frac{1}{(2\pi)^{2d}} }\int
F\circ\Phi^s_0e^{\bl{-iks}}ds\int e^{i(k'-k)t'}dt'
=F_{(k)}\delta_{kk'},
\]
we have that
\be\label{plo}
\norm{B}_\rho = \sum\limits_{\bl{k\in\Z^d}} \norm{B_{(k)}}_\rho \, e^{\rho|k|}.
\ee
%
% \sum\limits_{\substack{k\mbox{ \scriptsize such
%     that}\\-i<k,\omega>\in\LA}}\norm{B_{(k)}}_\rho e^{\rho|k|}$.
%
In particular, 
$\sum\limits_{\bl{k\in\Z^d}} \norm{B_{(k)}}_\rho \, e^{\eta|k|}$
%
% {\substack{k\mbox{ \scriptsize such
%      that}\\-i<k,\omega>\in\LA}}\norm{B_{(k)}}_\rho e^{\eta|k|}$
is convergent for each $\eta\leq\rho$.
 
%
%Noting that $|\lam|=|\scal{k}{\om}|$ and supposing 
Let us assume that the Diophantine
condition \eqref{diop} is satisfied.

\noindent By~\eqref{eq:decompBlam}, 
$\norm{B_{\lam}}_\rho \le \sum\limits_{ k\in\Z^d \mid i\scal{k}{\om} = \lam} 
\norm{B_{(k)}}_\rho$, hence, we have, using the definition~\eqref{eqepsrfinite},
\[
\epsilon_r \le \sum_{k_1,\ldots,k_r \in \Z^d}
\norm{B_{(k_1)}}_\rho \cdots \norm{B_{(k_r)}}_\rho
\, e^{\eta_r \rhotau_{\tau_r}\big( (i\scal{k_1}{\om}) \cdots (i\scal{k_r}{\om}) \big) }.
\]
Now, given $k_1,\ldots,k_r \in \Z^d$,  we have
\begin{eqnarray}
\rhotau_{\tau_r}\big( (i\scal{k_1}{\om}) \cdots (i\scal{k_r}{\om})
\big)
&=& \sum_{ \substack{ \sig\subset\{1,\ldots,r\} \\ \text{such that}\;\scal{\underline k_\sig}{\om}\neq0 } }
\frac1{ \abs{ \scal{\underline k_\sig}{\om} }^{1/\tau} }
\le \frac1{\al^{1/\tau}} 
\sum_{ \substack{ \sig\subset\{1,\ldots,r\} \\ \text{such that}\;\underline k_\sig\neq0 } }
\abs{ \underline k_\sig }
\nonumber\\
&\le& \frac{2^r}{\al^{1/\tau}} \big( 
\abs{k_1} + \cdots + \abs{k_r} \big).\nonumber
\end{eqnarray}
\ff{
We get, for $\eta_r2^r\alpha^{-1/\tau}\leq\rho$,
\begin{eqnarray}
\epsilon_r &\le& \sum_{k_1,\ldots,k_r \in \Z^d}
\norm{B_{(k_1)}}_\rho \cdots \norm{B_{(k_r)}}_\rho
\, e^{\frac{\eta_r2^r}{\alpha^{1/\tau}} \big( \abs{k_1} + \cdots + \abs{k_r} \big)}\nonumber\\
&\le& \sum_{k_1,\ldots,k_r \in \Z^d}
\norm{B_{(k_1)}}_\rho \cdots \norm{B_{(k_r)}}_\rho
\, e^{\rho \big( \abs{k_1} + \cdots + \abs{k_r} \big)}
=\norm{B}_\rho^r<\infty,\ \forall r\in\N^*.\label{epsr}
\end{eqnarray}}
%%
%and the \rhs\ coincides precisely with $\norm{B}_\rho^r$.
%

% we first note that, since $|\lam|=|\scal{k}{\om}|$,
%\[
%\sum\limits_{\lam_{i_1\dots i_m}\neq 0}|\lam_{i_1\dots i_m}|^{-1/\tau}
%\leq
%\alpha^{-1/{\tau}}\sum\limits_{\lam_{i_1\dots i_m}\neq 0}|k_{i_1}+\dots+k_{i_r}|
%\leq 2^{r-1}\alpha^{-1/\tau}\sum_{i=1}^r|k_i|,
%\]
%so that, for $\eta_r=\eta\alpha^{1/\tau}/2^{r-1}$,
%\begin{eqnarray}
%\epsilon_r&=&\sum\limits_
%%{\lambda_1,\dots\lambda_r\in\Lambda}
%{\underline\lambda=\lambda_1,\dots,\lambda_r\in\underline\Lambda}\norm{B_{\lambda_1}}_\rho\dots
%\norm{B_{\lambda_r}}_\rho 
%e^{\eta_r\rhotau_{\tau_r}(\underline\lambda) }\nonumber\\
%=\sum\limits_{\lam_1,\dots\lam_r\in\LA}\norm{B_{\lam_1}}_\rho\dots
%\norm{B_{\lam_r}}_\rho 
%e^{\eta_r\sum\limits_{m=1}^{r}\sum\limits_{\lam_{i_1\dots i_m}\neq 0}|\lam_{i_1\dots i_m}|^{-\tau}}&
%\leq& 
%\sum\limits_{\lam_1,\dots\lam_r\in\LA}\norm{B_{\lam_1}}_\rho\dots
%\norm{B_{\lam_r}}_\rho e^{\eta\sum\limits_{i=1}^r|n_i|}\nonumber\\
%&=&
%\left(\sum_\lam\norm{B_\lam}_\rho e^{\eta|n|}\right)^{r}
%\leq\norm{B}_\rho^r.
%\nonumber
%\end{eqnarray}

\ff{Therefore hypotheses 
%\eqref{eqBdecompBlam}--\eqref{eqepsrfinite} 
$(i)-(ii)$ are satisfied
 for all $B\in\J_\rho$, with $\eta_r=\rho\alpha^{1/\tau}2^{-r}$ and
$\tau_r=\tau$. 
%Therefore  
Theorem \ref{thmD} applies.}

Before to state it in the present setting, let us remark that since
$\epsilon_r\leq \norm{B}_\rho^r$, 
one can {improve~\eqref{resultA3} (or rather~\eqref{final})}. 
To do so we first remark that in Lemma \ref{RW}, if $\norm{B}_\rho\leq 1$, 
$$E_{N,\rho-\rho'}\leq\norm{B}_\rho
\sum_{r= 1}^{N}
\frac{(r-1)!}r \left(\frac\gamma{(\rho-\rho')^2}\right)^{r-1}
 G_r\left(\frac{\tau}{e\eta_r}\right)^{\tau r}
:=\norm{B}_\rho\Gamma_N.$$
 Therefore for $\epsilon=\inf{(1,\frac{\rho-\rho'}{8\gamma\Gamma_N})}$ we have that the second inequality of \eqref{final} is satisfied when
$
\norm{B}_\rho\leq \epsilon.
$

Under the same condition on $\norm{B}_\rho$ we find that 
$$
C^{sg}_{N+1}E_{N,(\rho-\rho')/2}^{N+1}\leq \norm{B}_\rho^{N+1}
C^{sg'}_{N+1}(\Gamma_N)^{N+1}
$$ 
with
$$
C_{N+1}^{sg'}=
2\left(\frac{(\rho-\rho')^2}{4\chi(\frac{\rho-\rho'}2)}+
1\right)\left(\frac {4\gamma}{(\rho-\rho')^2}\right)^{N+1},
$$ 
and 
$$
N^N(E_{N^2,\rho-\rho'}-E_{N+1,\rho-\rho'})\leq
\norm{B}_\rho^{N+1}N^N\Gamma_{N^2,N}
$$ 
with 
$$
\Gamma_{N^2,N}=\sum\limits_{r=N+1}^{N^2}
\frac{(r-1)!}r \left(\frac\gamma{(\rho-\rho')^2}\right)^{r-1}
 G_r\left(\frac{\tau}{e\eta_r}\right)^{\tau r},
$$ 
Thus we define:
\be\label{desp}
D=C^{sg'}_{N+1}(\Gamma_N)^{N+1}+2N^N\Gamma_{N^2,N}\mbox{ and }\epsilon=\inf\big{(1,\frac{\rho-\rho'}{8\gamma\Gamma_N}\big)}.
\ee
\vskip 1cm

Finally, $Z_N$ and $Y_N$ are real functions, and $e^{\ad_{Y_N}}$
corresponds to a composition by a symplectic transform. We get the
following rephrasing of Theorem \ref{thmD}.
 
 \begin{theorem}\label{thmDclass}
 Let $\rho>0$ such that $\norm{B}_\rho<\infty$. 
 
 For all $N\in\N^*$ and $0<\rho'<\rho$, let $D$ and $\epsilon$ be given by \eqref{desp}. Then
\begin{enumerate}[(a)]
\item 
the \bl{following two} expansions converge in $\mathcal B_{\rho'}$
\[
\tr{Z}{N} \defeq \sum_{r= 1}^{\bcbl N\ecb} \,\sum_{\lam_1,\ldots,\lam_r\in\LA}
\frac{1}{r} F^{\lam_1,\ldots,\lam_r} 
\{B_{\lam_r},\{\ldots\{B_{\lam_2},B_{\lam_1}\}\ldots\}\},
\]
 \[
\tr{Y}{N} \defeq\sum_{r=1}^{\bcbl N\ecb} \, \sum_{\lam_1,\ldots,\lam_r\in\LA}
\frac{1}{r} G^{\lam_1,\ldots,\lam_r} 
\{B_{\lam_r},\{\ldots\{B_{\lam_2},B_{\lam_1}\}\ldots\}\}
\]
\item 
\bl{and, if moreover $\norm{B}_\rho\leq \epsilon$, then}
\be\label{resultA3'class}
\left\{\begin{array}{l}
(X_0+B )\circ\Phi_N =X_0+Z_N+\cE_N,\\ 
\  \{X_0,Z_N\}=0,\\
\norm{\cE_N}_{\rho'}\leq D\norm{B}_\rho^{N+1},
\end{array}\right.
\ee
  $\Phi_N$ being the Hamiltonian flow at time $1$ of Hamiltonian $Y_N$.
  \end{enumerate}
\end{theorem}
\vskip 1cm
\begin{remark}\label{referee}
  In the two geometrical situations present in this section, namely
  $T^*\T^d$ and $T^*\R^d$, no use is made of the underlying symplectic
  structure. Therefore it seems to us reasonable to think that our
  methods apply to the situation of perturbations of Hamiltonian flows
  on Poisson manifolds. Indeed the method is essentially algebraic,
  using extensively the derivation $\ad$ referring only to the Poisson
  structure, so we are inclined to believe in the possibility of
  deriving a mould equation in this situation. The point then will to
  find a norm not using the Fourier transform (peculiar, say, to the
  linear or homogeneous spaces situation) but rather, and essentially
  equivalently, complex extensions of real analytic functions (in the
  case of real analytic Poisson manifolds).  More generally (and more
  difficult), it would be very interesting to transfer the methods of
  our paper to the question of the local description of a Poisson
  manifold around a symplectic leaf through the construction of normal
  forms as presented in \cite{MZ,MMZ,Mar}, or even to generalized
  complex geometry as in \cite{Ba}. We thank the referee for
  mentioning the possible extension of our work to Poisson geometry
  and pointing out the references quoted in this Remark.
\end{remark}
\section{Quantitative quantum formal normal forms}\label{quantintsysquantit}
 This section constitutes the quantum counterpart of the preceding section.

Let the Hilbert space $\cH$ be either $L^2(\R^d)$ and in this case let
$X_0=\frac12(-\hbar^2\Delta+\sum\limits_{i=1}^d\omega^2_ix_i^2)$ or
$L^2(\T^d)$ and $X_0=-i\hbar\omega\cdot\na$, corresponding 
indeed to the quantization of the two situations of Section
\ref{classintsystquantit}. 
Let us recall that in both cases $X_0$
is essentially self-adjoint on $\cH$.

 Here again we want to perturb $X_0$ by a ``small" perturbation $B$ and want to show that it is possible, after a conjugation by a unitary operator on $\cH$,  to put the new quantum Hamiltonian $X_0+B$ into a \textit{normal form} $X_0+Z,\ [X_0,Z]:=X_0Z-ZX_0=0$, modulo an error we want to be as small as we wish. 
%Here $\{\cdot,\cdot\}$ denotes the Poisson bracket on $\cP$.

 The result will be expressed in {Theorem~\ref{thmDquant} below} but let
 us first see how the quantum situations just mentioned enter also in
 the framework of the first part of this article, though they
 \fff{belong to} % seat on
 a very different paradigm than the one of the preceding section.

\bl{Let $J_\rho$ be the set of all pseudo-differential operators
whose Weyl symbols belong to $\mathcal J_\rho$.
We define the norm of an operator belonging to~$J_\rho$ as the $\Vert\cdot\Vert_\rho$ norm of its
symbol and we denote it by the same expression $\norm{\ \cdot\ }_\rho$.} 

There are different ways of defining Weyl quantization (see
Appendix \ref{allthat} below for elementary definitions). In the case
$\cP=T^*\R^d$,  one of them, actually the historical one
exposed in the book by Hermann Weyl \cite{weyl} consists in writing
again the formula \eqref{invfou} for the inverse Fourier transform 
\[
G(x,\xi)=\frac1{(2\pi)^{d}}\int \widehat G(q,p)e^{i({px-q\xi})}d\mu(q,p,)
\]
%
%\margeg{signes dans l'exp pour Fourier inverse?}
%
and replace in the right hand side $x$ and $\xi$ by $\times x$ and $-i\hbar\na$ respectively, in the case where $\P=T^*\R^d$. We get the operator ${B}$ associated to the symbol $\sig_{B}$ by the formula
\[
{B}=\frac1{(2\pi)^{d}}\int \widehat{\sig_{B}}(q,p)e^{i(px+i\hbar q\na)}d\mu(q,p,).
\]
The reader can check easily that when $\sig_{B}=x$ (resp. $\xi$) on recover ${B}=\times x$ (resp. $-i\hbar\na$). Moreover, using the Campbell-Hausdorff formula, one gets that
\[
e^{i(px+i\hbar q\na)}=e^{i\frac{pq}2}e^{ipx}e^{-q\hbar\na}.
\]
This is this formulation that we use in the case where $\P=T^*\T^{d}$ since  $px+i\hbar q\na$ doesn't make any sense on the torus, so we cannot use $e^{-i(px+i\hbar q\na)}$,  but $e^{-ipx}$ does (remember $p$ is the dual variable of $x$ so is discrete). Therefore we define in both cases
\be\label{defweyll}
{B}=\frac1{(2\pi)^{d}}\int \widehat{\sig_{B}}(q,p)
e^{i\frac{pq}2}e^{-ipx}e^{-q\hbar\na}
d\mu(q,p,).
\ee
Note that a straightforward computation gives back the usual formula
\eqref{weyl} or~\eqref{weylperiod} of Appendix~\ref{allthat}:
\[
{B} f(x)=\int \sig_{B}((x+y)/2,\xi)e^{-i\xi(x-y)/\hbar}f(y)dy\frac{d\xi}{(2\pi\hbar)^d}.
\]
But the main interest of this formula for our purpose is the fact that
$e^{i\frac{pq}2}e^{-ipx}e^{-q\hbar\na}$ is unitary, since
$e^{i\frac{pq}2}e^{-ipx}e^{-q\hbar\na}\varphi(x)=e^{i\frac{pq}2}e^{-ipx}\fff{\varphi(x-\hbar q)}$,
% $e^{i\frac{pq}2}e^{-ipx}e^{-q\hbar\na}\varphi(x)=e^{i\frac{pq}2}e^{-ipx}\varphi(x+\hbar q)$,
so
$\norm{e^{i\frac{pq}2}e^{-ipx}e^{-q\hbar\na}}_{L^2\to L^2}=1$\footnote{
We denote by $\norm{\cdot}_{L^2\to L^2}$ the operator norm on $\cH$.}
and therefore
\be\label{estn}
\norm{{B}}_{L^2\to L^2}\leq\norm{\widehat\sig_{B}}_{L^1(\P)}=``\norm{\sig_{B}}_0"\leq\norm{\sig_{B}}_\rho :=\norm{{B}}_\rho\mbox{ for all }\rho>0.
\ee
Moreover it is straightforward to show that when $\sig_{B}$ is real
valued, ${B}$ is a symmetric operator so that, when bounded, is
constitutes a symmetric bounded perturbation \bl{of~$X_0$}, therefore we
just proved the following result.
\begin{lemma}\label{nromess}
Let ${B}$ be defined by \eqref{defweyll} with $\sig_{B}$ real valued and let $\norm{{B}}_\rho:=\norm{\sig_{B}}_\rho<\infty$ for some $\rho>0$.
Then $X_0+{B}$ is essentially self-adjoint on $\cH$.
\end{lemma}

%\bl{Taking $\cB_\rho = J_\rho$ with Lie bracket 
%$\qu{\cdot\,}{\cdot} \defeq \frac1{i\hbar} [\cdot\,,\cdot]$
%%
%(where $[\cdot\,,\cdot]$ is the ordinary commutator),}
%%
%we are in the situation of an $X_0$-extended Banach scale of Lie
%algebras, by the same arguments as in Section \ref{classintsystquantit},
%%
%\bl{with $C=2e^2$ and $\chi(\rho)=e/\rho$}.

\bcg Let $\cL$ be the space of hermitian operators on $\cH$. 
The lemma in \fff{Appendix~\ref{appCommut} shows} % We just proved 
that $(\cL,[\cdot,\cdot]_q)$ endowed with $ J_\rho=\cB_\rho,\ 0<\rho<\infty$ and with $ [\cdot,\cdot]_q=\frac1{i\hbar}[\cdot,\cdot],\  \gamma=1\mbox{ and }\chi(\rho)=\frac1{e\rho}$, is an $X_0$-extended Banach scale  Lie algebra.\ec

\bcb 
The decomposition into $\lam-$homogeneous components of an arbitrary
$B\in J_\rho$ involves the letters of the same alphabet~$\LA$ defined
by~\eqref{eqdefLAom} as in Section~\ref{classintsystquantit}.
In fact, the homogeneous components of~$B$ can be obtained by
Weyl quantization of the homogeneous components of the symbol~$\sig_B$, or directly as
\ec
%
% (note that \eqref{bruck}
% $\Leftrightarrow \frac{[X_0,B_k]}{i\hbar}=\scal{k}{\om}
% B_k$ by the Diophantine property of $\omega$)
%
\be\label{bruck}
\bl{ B_\lam = \sum_{k\in\Z^d \mid i\scal{k}{\om} = \lam}\, }
%
% \sum\limits_{\substack{k\mbox{ \scriptsize such
%     that}\\-i<k,\omega>\in\LA}}\norm{B_{(k)}}_\rho e^{\rho|k|}$.
%
%\int e^{-itX_0/\hbar}Be^{itX_0/\hbar}e^{-ikt}dt
%
\bl{ \frac{1}{(2\pi)^d} } \int_{\T^d}
e^{ \frac{1}{i\hbar} t\cdot\overline X_0} \, B \,
e^{ -\frac{1}{i\hbar} t\cdot\overline X_0} \, e^{-ik\cdot t} \, dt
\qquad\text{for each $\lam\in\LA$,}
\ee
 where $\overline X_{0,j} = -\hbar^2\partial_{x_j}^2+\omega_j^2x_j^2$
 in the case of $\R^d$ and $\overline X_{0,j} =
 -i\hbar\omega_j\partial_{x_j}$ for $\T^d$, since linear Hamiltonian
 flows commute with quantization (see Lemma \ref{homweyl} below). 

 Therefore the hypothesis $(i)-(ii)$ of Theorem \ref{thmD} are satisfied for the
 same values 
% $\eta_r=\eta\alpha^{1/\tau}/2^{r-1}$ and $\tau_r=\tau$ 
 \bcbl
 \be\label{etatau}
 \eta_r=\rho\alpha^{1/\tau}2^{-r}\mbox{ and
}\tau_r=\tau
\ee
\ecb
as
 in Section \ref{classintsystquantit}.

 Moreover, if $Y_N$ is a self-adjoint operator, then $e^{\ad_{Y_N}}$
 corresponds to conjugation by the unitary transform $e^{\bl{\frac{1}{i\hbar}}Y_N}$,
hence the conclusions of Theorem \ref{thmD} for this situation can be
rephrased as:

 \begin{theorem}\label{thmDquant}
 Let $\rho>0$ such that $\norm{B}_\rho<\infty$. 
 
%Then f
For all $N\in\N^ *$ and $0<\rho'<\rho$, let $D$ and $\epsilon$ be given by \eqref{desp}. Then
\begin{enumerate}[(a)]
\item 
the \bl{following two} expansions converge in $\mathcal B_{\rho'}$
\[
\tr{Z}{N} \defeq \sum_{r= 1}^{\bcbl N\ecb} \,\sum_{\lam_1,\ldots,\lam_r\in\LA}
\frac{1}{r} F^{\lam_1,\ldots,\lam_r} 
\frac1{i\hbar}[B_{\lam_r},\frac1{i\hbar}[\ldots\frac1{i\hbar}[B_{\lam_2},B_{\lam_1}]\ldots]],
\]
 \[
\tr{Y}{N} \defeq\sum_{r=1}^{\bcbl N\ecb} \, \sum_{\lam_1,\ldots,\lam_r\in\LA}
\frac{1}{r} G^{\lam_1,\ldots,\lam_r} 
\frac1{i\hbar}[B_{\lam_r},\frac1{i\hbar}[\ldots\frac1{i\hbar}[B_{\lam_2},B_{\lam_1}]\ldots]]
\]
\item 
\bl{and, if moreover $\norm{B}_\rho\leq \epsilon$, then}
% and $U_N=e^{i\frac{Y_N}\hbar}$,
%
\be\label{resultA3'class}
\left\{\begin{array}{l}
e^{ \bl{\frac{1}{i\hbar} }Y_N}(X_0+B )
e^{-\bl{\frac{1}{i\hbar} } Y_N}=X_0+Z_N+\cE_N,\\ 
\  [X_0,Z_N]=0,\\
\norm{\cE_N}_{\rho'}\leq D\norm{B}_\rho^{N+1}.
\end{array}\right.
\ee

\end{enumerate}
\end{theorem}

\section{Semiclassical approximation}\label{secSemiCl2}
 In this final section we would like to link in a quantitative way the two preceding Section \ref{classintsystquantit} and \ref{quantintsysquantit}. Since
the estimates in Section \ref{quantintsysquantit} are uniform in the
Planck constant, it is natural to think that the quantum normal form
should be ``close" to the classical one when the Planck constant is
close to zero. Since such a comparison \fff{invokes} % invoke 
objects of different nature (operators for quantum, functions for classical), it is natural to use the symbol ``functor" $\sig$ to quantify this  link.
% Let us use the superscripts $C$ and $Q$ for denoting the ``same" object in Sections \ref{classintsystquantit} and \ref{quantintsysquantit} respectively, so that 
% $X_0^C(x,\xi)=\frac12(\xi^2+\sum\limits_{i=1}^d\omega_i^2x_i^2)$ or   $X_0^C(x,\xi)=\omega\cdot\xi$ and $X_0^Q=\frac12(-\hbar^2\Delta+\sum\limits_{i=1}^d\omega^2_ix_i^2)$   or $X_0^Q=-i\hbar\omega\cdot\na$. Note that in both cases, $\sig_{X_0^Q}=X_0^C$. So it seems natural to link $B^C$ and $B^Q$ by $\sig_{B^Q}=B^C$.

Expressing the quantum normal form in its mould-comould
expansion, we see that,  on one hand, the mould in independent of
$\hbar$, and (therefore) is the same as the one in the mould-comould
expansion of the classical normal form. On the other hand, for any pseudodifferential operator $B$, the
symbol of the commutator (divided by $i\hbar$) of any two homogeneous components of $B$ tends, by Lemma \ref{poiss}, to
the Poisson bracket of their two symbols, as $\hbar\to 0$. Moreover the symbols of
 such homogeneous components of $B$ are \fff{nothing but} % nothing else that
the homogeneous parts of the
symbol of $B$ by Corollary \ref{homweyl}. Finally, by
iteration of Lemma \ref{poiss}, iteration precisely estimated in
Proposition \ref{spero}, we see that the symbol of the quantum normal
form is, term by term in the mould-comould expansion, graduated by the
length of the words, close to the classical normal form, as
$\hbar\to 0$.

Our next result expresses quantitatively this fact, improving the
results of \cite{GP1} and \cite{DGH}.  
%Let us recall $Z^Q_N$
%(resp. $Z^C_N$) is the quantum (resp. classical) normal form in
%Theorem \ref{thmDquant} (resp. Theorem \ref{thmDclass}).  

\bcbl
For $N\geq 1$, we will  denote  by $Z^Q_N$
(resp. $Z^C_N$)   the quantum (resp. classical) normal form of $X_0^Q+B^Q$ (resp. $X_0^C+B^C$)  as expressed in
Theorem \ref{thmDquant} (resp. Theorem \ref{thmDclass}). Here $X_0^Q=\frac12(-\hbar^2\Delta+\sum\limits_{i=1}^d\omega^2_ix_i^2)$ or
$-i\hbar\omega\cdot\na$ as in Section \ref{quantintsysquantit} and $X_0^C(x,\xi)=\frac12\Big(
\sum\limits_{j=1}^d\xi_j^2+\sum\limits_{j=1}^d\omega_j^2x_j^2 \Big)$
or $\omega\cdot\xi$ as in Section \ref{classintsystquantit}. Note that, in both cases, $X_0^Q=\mbox{Op}^{W}(X_0^C)$. Let us recall that $\omega$ satisfies the Diophantine condition \eqref{diop} with parameters $\alpha,\tau$.
\begin{theorem}\label{corBNF}
%Let $N\geq 1$. 
%and let us denote $Z^Q_N$
%(resp. $Z^C_N$) is the quantum (resp. classical) normal form associated to the $B^Q$ (resp. $B^C$) in
%Theorem \ref{thmDquant} (resp. Theorem \ref{thmDclass}).
Let us suppose that $B^Q=\mbox{Op}^{W}(B^C)$, so that $X_0^Q+B^Q=\mbox{Op}^{W}(X_0^C+B^C)$.

Then $Z^Q_N=\mbox{Op}^{W}(\sig_{ Z^Q_N})$ where, for all $N\geq 2$, $\sig_{ Z^Q_N}$ satisfies, for $\rho'<\rho$,
\[
\norm{\big(\sig_{ Z^Q_N}-\sig_{Z^Q_{N-1}}\big)-(Z^C_N-Z^C_{N-1})}_{\rho'}
\leq \hbar^2 C_N\norm{B}_\rho^N
\]
where $C_N=\frac{F_N}{6N}\left(\frac{2^{N}\tau}{e\rho\alpha^{1/N}}\right)^{(N-1)\tau}
\left(\frac{N+2}{e(\rho-\rho')}\right)^{N+2}$ and $F_N$ is defined by \eqref{frgr}.
\end{theorem}
%\bcbl
Note that, when $N=1$, $Z^Q_1=B^Q_0=\mbox{Op}^{W}(B^C_0)$ and $Z^C_1=B^C_0$, so that $\sig_{Z^Q_1}-Z^C_1=0$.
\begin{proof}
By Theorem \ref{thmDquant} $(a)$, we have that
\[
Z^Q_N=
\sum_{r= 1}^{N}
\sum_{\lam_1,\ldots,\lam_r\in\LA}
\frac{1}{r} F^{\lam_1,\ldots,\lam_r} 
\frac1{i\hbar}[B^Q_{\lam_r},\frac1{i\hbar}[\ldots\frac1{i\hbar}[B^Q_{\lam_2},B^Q_{\lam_1}]\ldots]],
\] 
\fff{therefore} % Therefore 
$Z^Q_N=\mbox{Op}^{W}(\sig_{Z^Q_N})$ with
%with, by linearity of the ``functor" $\sig$ and  Lemma \ref{poiss} 
%and definition \eqref{twist}
$
\sig_{ Z^Q_N}=
\sum\limits_{r= 1}^{N}\frac{1}{r} F^{\lam_1,\ldots,\lam_r}
%\sig_{B^Q_d}*_{\hbar}
(\sig_{B^Q_{\lam_r}}*_{\hbar}(\dots*_{\hbar}(\sig_{B^Q_{\lam_2}}*_{\hbar}\sig_{B^Q_{\lam_1}})))
$\\
 by Lemma~\ref{poiss}, and
\[
%\sig_{\big( Z^Q_N-Z^Q_{N-1}\big)}=
\sig_{ Z^Q_N}-\sig_{Z^Q_{N-1}}=
\sum_{\lam_1,\ldots,\lam_N\in\LA}\frac{1}{N} F^{\lam_1,\ldots,\lam_N}
\sig_{B^Q_{\lam_N}}*_{\hbar}(\sig_{B^Q_{\lam_{N-1}}}*_{\hbar}(\dots*_{\hbar}(\sig_{B^Q_{\lam_2}}*_{\hbar}\sig_{B^Q_{\lam_1}}))).
\]
On the other hand, by Theorem \ref{thmDclass} and for the same coefficients $F^{\lam_1,\ldots,\lam_N}$,
\[
Z_N^C-Z_{N-1}^C=
\sum_{\lam_1,\ldots,\lam_N\in\LA}
\frac{1}{N} F^{\lam_1,\ldots,\lam_N} 
\{B^C_{\lam_N},\{\ldots\{B^C_{\lam_2},B^C_{\lam_1}\}\ldots\}\}.
\]
Since $B^Q=\mbox{Op}^{W}(B^C)$ we have, by Corollary \ref{homweyl}, that $\sig_{B^Q_\lam}=B^C_\lam$ so that, by Proposition \ref{spero},
\[
\norm{\big(\sig_{ Z^Q_N}-\sig_{Z^Q_{N-1}}\big)-(Z^C_N-Z^C_{N-1})}_{\rho'}
\leq \hbar^2 \frac{1}{6N}\left(\frac{N+2}{e(\rho-\rho')}\right)^{N+2}
\sum_{\lam_1,\ldots,\lam_r\in\LA}
 F^{\lam_1,\ldots,\lam_N} 
\prod_{i=1}^N\norm{B^C_{\lam_i}}_\rho
\]
%with
%$
%C'_N=
%%\frac16
%\sum\limits_{\lam_1,\ldots,\lam_r\in\LA}
%%
%\frac{1}{6N} F^{\lam_1,\ldots,\lam_N} 
%%\frac{\hbar^2}6
%\left(\frac{N+2}{e(\rho-\rho')}\right)^{N+2}
%$.
%%F_r
%%%
%%\left(\frac{{\tau_r}}{e\eta_r}\right)^{(r-1)\tau_r}
%%%
%%e^{ \eta_r \rhotau_{\tau_r}(\ula)},
Using now the first estimate of  Corollary \ref{corconj} with $(\eta_r,\tau_r)$ given by \eqref{etatau}, and the definition \eqref{deff}, we get
\[
\sum_{\lam_1,\ldots,\lam_N\in\LA}
 F^{\lam_1,\ldots,\lam_N} 
\prod_{i=1}^N\norm{B^C_{\lam_i}}_\rho
\leq
F_N
\big(\frac\tau{e\rho\alpha^{\frac1N}2^{-N}}\big)^{(N-1)\tau}\eps_N
\leq 
F_N
\big(\frac\tau{e\rho\alpha^{\frac1\tau}2^{-N}}\big)^{(N-1)\tau}
\norm{B}_\rho^N,
\]
where we have used \eqref{epsr} for the last inequality.
\end{proof}

%The proof of Theorem \ref{corBNF} is easily obtained by using the mould expansion of $Z^Q_N$, in particular
%\[
%Z^Q_N-Z^Q_{N-1}=
%\sum_{\lam_1,\ldots,\lam_N\in\LA}
%%
%\frac{1}{N} F^{\lam_1,\ldots,\lam_N} 
%%
%\frac1{i\hbar}[B_{\lam_N},\frac1{i\hbar}[\ldots\frac1{i\hbar}[B_{\lam_2},B_{\lam_1}]\ldots]],
%\] 
%the linearity of the ``functor" $\sig$, Corollary \ref{corconj}, Proposition \ref{spero}, and the arguments expressed just before the statement of the theorem.
%Note that, when $N=1$, $Z^Q_1=B^Q_0=\mbox{Op}^{W}(B^C_0)$ and $Z^C_0=B^C_0$, so that $\sig_{Z^Q_1}-Z^C_1=0$.

\begin{remark}
\fff{Theorem~\ref{corBNF} implies}
% We could have expressed the result on the form 
$\norm{\sig_{ Z^Q_N}-Z^C_N}_{\rho'} \leq
\hbar^2 C'_N \fff{\norm{B}_\rho^2 ( 1 - \norm{B}_\rho )^{-1}}$
\fff{with $C'_N = \max\{C_2,\ldots,C_N\}$,}
%
% \hbar^2C'_N\norm{B}_\rho\frac{1-\norm{B}_\rho^N}{1-\norm{B}_\rho}$, 
but this \fff{is less precise than the result stated above}.
% but this would have been  less precise as it is stated in Theorem \ref{corBNF}.
\end{remark}
\ecb
Note that the correction is of order $2$ in the Planck constant, which
means that the classical perturbation theory incorporates the entire
Bohr-Sommerfeld quantization, including the Maslov index.

%%%%%%%%%%%%%%%%%%%%%%%%%%%%%%%%%%%%%%%%%%%%%%%%%%%%%%%%%%%%%%%%%
%%%%%%%%%%%%%%%%%%%%%%%%%%%%%%%%%%%%%%%%%%%%%%%%%%%%%%%%%%%%%%%%%

% \delta_{0}\cdots\delta_{d-2}(\frac{\delta}{d})^{d-1}}\prod_{i=1}^{d-1}\Vert X_i  \Vert_{\rho}\Vert G_{1} \Vert_{\rho-\delta_{d-1} }\nonumber\\
% &\leq & \frac{C^{d}}{d! \delta_{0}\cdots\delta_{d-1}\delta_{d-1}(\frac{\delta}{d})^{d-1}}\prod_{i=1}^{d}\Vert X_i \Vert_{\rho}\Vert Y \Vert_{\rho }\nonumber\\
% &\leq & \frac{C^{d}}{d! d!(\frac{\delta}{d})^{d} \delta_{d-1}(\frac{\delta}{d})^{d-1}}\prod_{i=1}^{d}\Vert X_i \Vert_{\rho}\Vert Y \Vert_{\rho }\nonumber\\
% &\leq & \frac1{2\pi}\left(\frac{Cd^2}{\delta^2}\right)^d\frac1{ (d-1)!d!}\prod_{i=1}^{d}\Vert X_i \Vert_{\rho}\Vert Y \Vert_{\rho }\nonumber\\
% \norm{B}_{\rho''}\right)\frac{\left(\frac
% \gamma{(\rho''-\rho')^2}\norm{Y_N}_{\rho''}\right)^{N+1}}{\left( 1-\frac  \gamma{(\rho''-\rho')^2}\norm{Y_N}_{\rho''}\right)}\nonumber\\

%%%%%%%%%%%%%%%%%%%%%%%%%%%%%%%%%%%%%%%%%%%%%%%%%%%%%%%%%%%%%%%%%
%%%%%%%%%%%%%%%%%%%%%%%%%%%%%%%%%%%%%%%%%%%%%%%%%%%%%%%%%%%%%%%%%

%\newpage

\addtocontents{toc}{\vspace{.5em}}

%\vspace{1.5cm}

%%%%%%%%%%%%%%%%%%%%%%%%%%%%%%%%%%%%%%%%%%%%%%%%%%%%%%%%%%%%%%%%%
%%%%%%%%%%%%%%%%%%%%%%%%%%%%%%%%%%%%%%%%%%%%%%%%%%%%%%%%%%%%%%%%%

\begin{appendix}
\section{Weyl quantization and all that}\label{allthat}

Weyl quantization 
has been defined in Section \ref{quantintsysquantit}. Defining the
unitary operator 
%
%\margeg{mettre $(q,p)$ (comme pr\'ec\'edemment) plut\^ot que $(p,q)$}
%
$U(q,p)$ where $(q,p)$ are the Fourier variables of $\cP=T^*\R^d\mbox{ or }T^*\T^d$ by:
\[
U(q,p)\phi(x)=e^{i\frac{pq}2}e^{ipx}e^{-q\hbar\na}\fff{\phi(x)}=e^{i\frac{pq}{2}}e^{i{px}}\phi(x-\hbar q),\ \phi\in L^2(\R^d\mbox{ or }\T^d),
\]
the Weyl quantization of a function $\sig_V$ on $\cP$ is the operator
\be\label{weylfourier} 
{V}=\frac1{(2\pi)^d}\int dpdq \widehat{\sig_{V}}(q,p)U(q,p):=\mbox{Op}^{W}(\sig_V)
\ee
%
%\margeg{v\'erifier $x\leftrightarrow q$, \; $\xi\leftrightarrow p$}
%
where $dpdq$ is used for $d\mu(q,p)$ as in Section \ref{classintsystquantit} page~\pageref{dpdq} and $\widehat\cdot$ is the symplectic Fourier transform defined by \eqref{fou}\footnote{The goal of this appendix is not to give a crash course on pseudo-differential operators, but rather to recall the strict minimum used in the present paper. 
The reader is referred to \cite{Fo} for a general exposition. The reader not familiar with the presentation here can recognize easily the Weyl quantization of a symbol $\sig^{}_{V}$  being, e.g. of the Schwartz class. That is to say that, when $\sig^{}_{V}\in\mathcal S(\R^{2d})$, ${V}$ defined by \eqref{weylfourier} acts on a function $\varphi\in L^2(\R^d)$ through the formula
\be\label{weyl}
{V}\varphi(x)=\int_{\R^d\times \R^d}\sig^{}_{V}\big(\frac{x+y}2,\xi\big)e^{-i\frac{\xi(x-y)}\hb}\varphi(y)\frac{d\xi dy}{(2\pi\hb)^d}.
\ee
and, in the case where $\sig^{}_{V}\in C^\infty(\T^d)\otimes\mathcal S (\R^d)$, $V$ acts on a function $\varphi\in L^2(\T^d)$ by the same formula
\be\label{weylperiod}
{V}\varphi(x)=\int_{\R^d\times \R^d}\sig^{}_{V}\big(\frac{x+y}2,\xi\big)e^{-i\frac{\xi(x-y)}\hb}\varphi(y)\frac{d\xi dy}{(2\pi\hb)^d}
\ee
where, in \eqref{weylperiod}, it is understood that $\sig^{}_{V}(\cdot,\xi)$ and $\varphi$ are extended to $\R^d$ by periodicity (see \cite{PS}). Note that \eqref{weyl} and \eqref{weylperiod} make sense thanks to the Schwartz property of ${V}$ in $\xi$ and that in \eqref{weylperiod} the r.h.s. depends only on the values of $\sig_{{V}}(x,\xi)$ for $\xi\in\hbar\Z^d$.}.

Obviously, as mentioned earlier, $\norm{U(q,p)}_{L^2\to L^2}=1$ and therefore
\be\label{normss}
\norm{{V}}_{L^2\to L^2}
\leq\norm{\sig_{V}}\rho,\ 
\forall \rho>0.
\ee
Note that \eqref{weylfourier} makes also sense when $\sig_V$ is a polynomial on $\cP$ (polynomial in the variable $\xi$ in the case $\cP=T^*\T^d$) since $U(q,p)$ as defining an unbounded operator. One check easily that this is the case for $X_0$ in the two examples of Section \ref{classintsystquantit} and \ref{quantintsysquantit}.

The following result is the fundamental one concerning the transition quantum-classical and, as presented here, is the only one we really need in the present article.
%Its proof is straightforward for symbols in the Schwartz class, by using \eqref{weyl}, or \eqref{weylperiod}, and \eqref{symbweyl}. It gives a mod($\hbar)$-homomorphism between quantum and classical Lie algebras.
\begin{lemma}\label{poiss}
Let $V=\mbox{Op}^{W}(\sig_V),V'=\mbox{Op}^{W}(\sig_{V'})$ with $\sig_{V},\sig_{V'}$ either belong to $\mathcal J_\rho$ for a certain $\rho>0$ or are polynomials on $\cP$.

Then $$\frac1{i\hbar}[V,V']=\mbox{Op}^{W}(\sig_{V}*_{\hbar}\sig_{V'}))$$
where $*_{\hbar}$ is defined through the Fourier transform by
\be\label{twist}
\widehat{\sig*_{\hbar}\sig'}(q,p)=\int_\cP
\frac{\sin{\big[\hbar((q-q')p'-(p-p')q')\big]}}\hbar\widehat\sig(p-p',q-q')\widehat{\sig'}(p',q')dp'dq'.
\ee

In particular
\[
\lim_{\hbar\to 0}\sig *_{\hbar}\sig' =\{\sig,\sig'\}.
\]
and
\[
\sig_{2} *_{\hbar}\sig'=\{\sig_{2},\sig'\}
\]
when $\sig_{2}$ is a quadratic form. Therefore, \bl{in} this case,
\be\label{therefore}
\exp\big({\ad_{\mbox{Op}^W(\sig_2)}}\big)\mbox{Op}^W(\sig_V) =
\mbox{Op}^W(e^{ \bl{ \ad_{\sig_2} } }\sig_V).
\ee
\end{lemma}

Using \eqref{bruck} and \eqref{therefore} we get the following result.

\begin{corollary}\label{homweyl}
Let $X_0$ be as in Section~\ref{quantintsysquantit}. 
Then the homogeneous component
$B_\lam,\ \lam=\scal{k}{\om},\\ k\in~\Z^d$ of any
pseudodifferential operator $B$ is the Weyl quantization 
\bl{of the $\lam-$homogeneous part (with respect to $\sig_{X_0}$) % $\sig^{}_{{B}_\lam}$
of $\sig^{}_{B}$}, that is
\[
%\sig^{}_{V_\lam}=\int_{\T^d}\sig^{}_V\circ R_\theta d\theta
B_\lam=\mbox{Op}^W(
%\int_{\T^d}\sig^{}_B\circ R_\theta d\theta
(\sig_B)_\lam).
\]

\end{corollary}
\vskip 1cm
Using the operator $A$ of Lemma \ref{poiss} and similar arguments that the ones used in the proof of Lemma \ref{psto} we get the following result.
\begin{proposition}\label{spero}
Let $0<\rho'
%<\rho''
<\rho$. Then
for any $d\geq 2$ 
\begin{eqnarray}
\norm{
\sig_{B_d}*_{\hbar}(\sig_{B_{d-1}}*_{\hbar}(\dots*_{\hbar}(\sig_{B_2}*_{\hbar}\sig_{B_1})))
-\{\sig_{B_d},\{\sig_{B_{d-1}},\dots\{\sig_{B_2},\sig_{B_1}\}\}\}}_{\rho'}\nonumber\\
\leq
\frac{\hbar^2}6
\left(\frac{d+2}{e(\rho-\rho')}\right)^{d+2}\prod_{k=1}^d
\norm{B_k}_\rho.\nonumber
\end{eqnarray}
\end{proposition}
 Note that 
\[
\frac{[B_d,[B_{d-1},\dots[B_2,B_1]]]}{(i\hbar)^d}=
\mbox{Op}^W(\sig_{[B_d,[B_{d-1},\dots[B_2,B_1]]]/(i\hbar)^d})
\]
with $\sig_{[B_d,[B_{d-1},\dots[B_2,B_1]]]/(i\hbar)^d}=
\sig_{B_d}*_{\hbar}(\sig_{B_{d-1}}*_{\hbar}(\dots*_{\hbar}(\sig_{B_2}*_{\hbar}\sig_{B_1})))$.
\begin{proof}
The proof will be using the methods of the one of Lemma \ref{tout}. 

Iterating \eqref{compfou} we get
\begin{eqnarray}\label{compfou}
\widehat\sig_{[B_d,[B_{d-1},\dots[B_2,B_1]]]/(i\hbar)^d}(p_d,q_d)=\int dp_1dq_1\dots dp_{d-1}dq_{d-1}\nonumber\\
\frac1\hbar\sin{{\hbar}((q_d-q_{d-1})p_{d-1}-(p_d-p_{d-1})q_{d-1})}\widehat\sig_{B_d}(p_d-p_{d-1},q_d-q_{d-1})\nonumber\\
\frac1\hbar\sin{{\hbar}((q_{d-1}-q_{d-2})p_{d-2}-(p_{d-1}-p_{d-2})q_{d-2})}\widehat\sig_{B_{d-1}}(p_{d-1}-p_{d-2},q_{d-1}-q_{d-2})\nonumber\\
\dots\nonumber\\
\frac1\hbar\sin{{\hbar}((q_{3}-q_{2})p_{2}-(p_{3}-p_{2})q_{2})}\widehat\sig_{B_{3}}(p_{3}-p_{2},q_{3}-q_{2})\nonumber\\
\frac1\hbar\sin{{\hbar}((q_{2}-q_{1})p_{1}-(p_{2}-p_{1})q_{1})}\widehat\sig_{B_{2}}(p_{2}-p_{1},q_{2}-q_{1})\nonumber\\
\widehat\sig_{B_{1}}(p_{1},q_{1}).\nonumber
%\widehat{\sig_ F}(q-q',p-p')\widehat\sig_{G}(q',p')dq'dp'
\end{eqnarray}
Expanding 
$$\prod\limits_{i=1}^d\frac1\hbar\sin{{\hbar} x_i}=\prod\limits_{i=1}^dx_i-\frac{\hbar^2}6
\sum\limits_{k=1}^d{x^3_k}\sin{\hbar'x_k}\prod\limits_{l\neq k}
\frac1\hbar\sin{{\hbar'} x_l}$$ 
for some $0\leq\hbar'\leq\hbar$,
 one realizes that the first term gives precisely after integration the Fourier transform of $\{\sig_{B_d},\{\sig_{B_{d-1}},\dots\{\sig_{B_2},\sig_{B_1}\}\}\}$. 
 
Using   $|\sin{\hbar'x_k}|\leq 1$ and $|\frac1\hbar\sin{\hbar'x_l}|\leq |x_l|$ we get
\begin{eqnarray}
\norm{\sig_{[B_d,[B_{d-1},\dots[B_2,B_1]]]/(i\hbar)^d}-\{\sig_{B_d},\{\sig_{B_{d-1}},\dots\{\sig_{B_2},\sig_{B_1}\}\}\}}_{\rho'}
\nonumber\\
\leq \frac{\hbar^2}6
\int e^{\rho'(|q_1|+|p_1|+\dots+|q_d|+|p_d|)}dp_ddq_d
dp_{d-1}dq_{d-1}\dots dp_1dq_1\nonumber\\
\sum\limits_{k=1}^d{|x_k|^3}\prod\limits_{l\neq k}
|x_l|
|\widehat\sig_{B_k}(p_k-p_{k-1},q_k-q_{k-1}||\widehat\sig_{B_l}(p_l-p_{l-1},q_l-q_{l-1}|
\nonumber
\end{eqnarray}
with $p_{-1}=q_{-1}:=0$ and $x_k:=(p_k-p_{k-1})q_{k-1}-(q_k-q_{k-1})p_{k-1}=p_kq_{k-1}-q_kp_{k-1}$.

We have obviously that 
\[
\sum\limits_{k=1}^d{|x_k|^3}\prod\limits_{l\neq k}
|x_l|\leq \sum_{\substack{(a_1,\dots,a_m)\\
\in\{|p_1|,|q_1|,\dots,|p_d|,|q_d|\}{d+2}}}\prod_{m=1}^{d+2}a_m
\leq
(|p_1|+|q_1|+\dots+|p_d|+|q_d)^{d+2}
\]
Therefore, using a last time the magic tool $x^\beta e^{-\eta x}\leq\left(\frac\beta{e\eta}\right)^\beta,\ \beta,\eta,x\geq 0$, we get that 
\be\label{estifi}
\sum\limits_{k=1}^d{|x_k|^3}\prod\limits_{l\neq k}
|x_l|
e^{\rho'(|q_1|+|p_1|+\dots+|q_d|+|p_d|)}
\leq 
\left(\frac{d+2}{e(\rho-\rho')}\right)^{d+2}
e^{\rho(|q_1|+|p_1|+\dots+|q_d|+|p_d|)}.
\ee
Defining $P_k=p_k-p_{k-1}$ and $Q_k=q_k-q_{k-1}$ and using
\begin{eqnarray}
|q_d|+|p_d|\leq |q_d-q_{d-1}|+|q_{d-1}-q_{d-2}|+\dots |q_2-q_1|+|q_1|+|p_{d}-p_{d-1}|+\dots+|p_1|)\nonumber\\
=\sum_{k=1}^d(|P_k|+|Q_k|)\nonumber
\end{eqnarray}
 in $e^{\rho(|q_1|+|p_1|+\dots+|q_d|+|p_d|)}$, we get the result by \eqref{estifi} and 
the change of variables $(p_k,q_k)\to(P_k,Q_k)$ (note that the covariance property with respect to the flow generated by $X_0$ is exactly the same as explained in the beginning of the proof of Lemma \ref{tout}).
\end{proof}

\section{Estimating Lie brakets}   \label{appLieBr}
The following Lemma is {a slight generalization of \cite[inequality~(5.11)]{PS}} (see also
\cite{GP}).
% as it is clear in the proof of this Proposition that
% (5.11) uses only (5.10) which becomes one of our assumption.

\begin{lemma}\label{psto}
Let us suppose that for $0<\rho'<\rho''<\rho$ and $i=1\dots d, d\in\N^*$
\be\label{hyplie}
\norm{[X_i,Y]}_{\rho'} \le \frac{\gamma}{e^2(\rho-\rho')(\rho''-\rho')} \norm{X_i}_{\rho} \norm{Y}_{\rho''},\ \Vert[X_0,X_i]\Vert_{\rho'}\leq \frac1{\chi(\rho-\rho')}\Vert X_i\Vert_\rho\  .
\ee
Then, 
%denoting $\gamma:=Ce^2$,
\be\label{conclie}
\frac1{d!}\norm{
[ X_d,[X_{d-1},\dots[X_1
%_{d\ times}
,Y]]]}_{\rho'}\leq \frac{\gamma^d}{(\rho-\rho')^{2d}}\norm{Y}_\rho\prod_{i=1}^d\norm{X_i}_\rho
\ee
 and 
\be\label{concliex0}
\frac1{d!}\norm{[ X_d,[X_{d-1},\dots[X_1,X_0]]]}_{\rho'}\leq  \frac{\gamma^d}{(\rho-\rho')^{2d}}\frac{(\rho-\rho')^2}{\chi(\rho-\rho')}\prod_{i=1}^d\norm{X_i}_\rho.
\ee
\end{lemma}
Writing $e^{\ad_X}Y=\sum\limits_{d=0}^\infty\frac1{d!}\underbrace{[X,[X,\dots[X}_{d\ times},Y]]]$ we get easily the following Corollary.
\begin{corollary}\label{lemexp}
\[
\norm{e^{\ad_X}Y}_{\rho'}\leq\frac{\norm{Y}_\rho}{1-\frac \gamma{(\rho-\rho')^2}\norm{X}_\rho}
\]
\and
\[
\norm{e^{\ad_X}X_0-X_0}_{\rho'}\leq\frac{\gamma\norm{X}_\rho}{\chi(\rho-\rho')\left(1-\frac \gamma{(\rho-\rho')^2}\norm{X}_\rho\right)}
\]
Moreover
\[
\norm{e^{\ad_X}Y-\sum\limits_{d=0}^N\frac1{d!}\underbrace{[X,[X,\dots[X}_{d\ times},Y]]]}_{\rho'}\leq\norm{Y}_\rho\frac{\left(\frac \gamma{(\rho-\rho')^2}\norm{X}_\rho\right)^{N+1}}{1-\frac \gamma{(\rho-\rho')^2}\norm{X}_\rho}
\]
\and
\[
\norm{e^{\ad_X}X_0-X_0-\sum\limits_{d=1}^N\frac1{d!}\underbrace{[X,[X,\dots[X}_{d\ times},X_0]]]}_{\rho'}\leq\frac{(\rho-\rho')^2}{\chi(\rho-\rho')}\frac{\left(\frac \gamma{(\rho-\rho')^2}\norm{X}_\rho\right)^{N+1}}{1-\frac \gamma{(\rho-\rho')^2}\norm{X}_\rho}
\]
\end{corollary}
\begin{proof}[Proof of Lemma \ref{psto}](following \cite{GP} and \cite{PS} where the case $X_i=X$ is studied)
\eqref{conclie} is easily obtained by iteration of the first part of \eqref{hyplie}.
%(see \cite{GP} Section 3.3 for details). 
Consider the finite sequence of numbers $\delta_s=\frac{d-s}{d}\delta$. We have $\delta_0=\delta$, $\delta_d=0$ and $\delta_{s-1}-\delta_s=\frac{\delta}{d}$. Let us define $G_0:= Y$ and $G_{s+1}:=[X_{s+1},G_s]$, for $0\leq s \leq d-1$.
According to \eqref{hyplie}, we have,
 %gives that, for all $G_s,G_{s+1}$, 
 denoting $C=\frac\gamma{e^2}$,
\[
\norm{G_{s}}_{\rho-\delta_{d-s}} \leq
\frac{C}{\de_{d-s}(\frac{\de}{d})}
\norm{X_s}_{\rho} \norm{G_{s-1}}_{\rho-\de_{d-s+1}}
\quad\text{for $1\le s \le d$.}
\]

Hence, by induction, we obtain, since $\delta_0=\delta$ and $G_0:=Y$,

\begin{eqnarray}
\frac{1}{d!}\Vert G_{d} \Vert_{\rho-\delta_{0}} &\leq &
%\frac{C^{d-1}}{d! \delta_{0}\cdots\delta_{d-2}(\frac{\delta}{d})^{d-1}}\prod_{i=1}^{d-1}\Vert X_i \Vert_{\rho}\Vert G_{1} \Vert_{\rho-\delta_{d-1} }\nonumber\\
%&\leq & \frac{C^{d}}{d! \delta_{0}\cdots\delta_{d-1}\delta_{d-1}(\frac{\delta}{d})^{d-1}}\prod_{i=1}^{d}\Vert X_i \Vert_{\rho}\Vert Y \Vert_{\rho }\nonumber\\
\frac{C^{d}}{d! \delta_{0}\cdots\delta_{d-1}
%\delta_{d-1}
(\frac{\delta}{d})^{d}}\prod_{i=1}^{d}\Vert X_i \Vert_{\rho}\Vert Y \Vert_{\rho }\nonumber\\
&\leq & \frac{C^{d}}{d! d!(\frac{\delta}{d})^{d} 
%\delta_{d-1}
(\frac{\delta}{d})^{d}}\prod_{i=1}^{d}\Vert X_i \Vert_{\rho}\Vert Y \Vert_{\rho }\nonumber\\
&\leq & 
\left(\frac{Cd^2}{\delta^2}\right)^d\frac1{d!d!}\prod_{i=1}^{d}\Vert X_i \Vert_{\rho}\Vert Y \Vert_{\rho }\nonumber\\
&=&\frac1{2\pi d}\left(\frac{\gamma}{\delta^2}\right)^d
\left(\frac{\sqrt{2\pi d}d^de^{-d}}{ d!}\right)^2
%\Vert 
\prod_{i=1}^{d}\Vert X_i \Vert_{\rho}\Vert Y \Vert_{\rho }\nonumber\\
&\leq&\frac1{2\pi d}\left(\frac{\gamma}{\delta^2}\right)^d
\prod_{i=1}^{d}\Vert X_i \Vert_{\rho}\Vert Y \Vert_{\rho }<\left(\frac{\gamma}{\delta^2}\right)^d
\prod_{i=1}^{d}\Vert X_i \Vert_{\rho}\Vert Y \Vert_{\rho }.\nonumber
\end{eqnarray}

The proof of \eqref{concliex0} follows exactly the same lines.
\end{proof}

\section{Estimating Poisson brackets} \label{appPoisson}
\begin{lemma}\label{poisson}
\[
\norm{\{F,G\}}
_{\rho'}\leq 
%\frac{2}{e^2(\rho-\rho')^2}
%\norm{F}_\rho\norm{G}_\rho
\frac{\bcg1\ec}{e^2(\rho-\rho')(\rho''-\rho')}
\norm{F}_\rho\norm{G}_{\rho''}
\]
whenever $\rho'<\rho''\leq \rho$.
\end{lemma}
\bcg

\begin{proof}
We will first prove, in the two cases 
$\cP = \R^d\times\R^d$
  and $\cP=T^*\T^d$ the following identity.
  \be
\widehat{\{F,G\}_{(k)}}(q,p)=
\sum_{k'\in\Z^d}\int dp'dq'
((q-q')p'-(p-p')q')
\widehat{F_{(k-k')}}(q-q',p-p')\widehat{G_{(k')}}(q',p').\label{sin}
\ee
\begin{proof}[Proof of \eqref{sin}.]

$\{F,G\}(x,\xi)=\partial_{\xi}F(x,\xi)\partial_x G(x,\xi)-
\partial_x F(x,\xi)\partial_\xi G(x,\xi)$. So 
\begin{eqnarray}\label{class}
{\widehat{\{F,G\}}}(q,p)&=&\int ((q-q')p'-(p-p')q')\widehat F(q-q',p-p')\widehat{G}(q',p')dq'dp'\nonumber\\
&=&\int(q-q',p-p')\cdot(p',q') \widehat
    F(q-q',p-p')\widehat{G}(q',p')dq'\fff{dp'}. % dp'
\end{eqnarray}

In the case $\cP = \R^d\times\R^d$, $\widehat{\{F,G\}_{(k)}}=\widehat{\{F,G\}}_{(k)}$ so, since $\Phi^t_0$ is linear symplectic and so preserves Liouville measure, we get by using \eqref{class},
\begin{eqnarray}
\widehat{{\{{F,G}\}}_{(k)}}(q,p)\nonumber\\
=
\frac1{(2\pi)^d}\int dtdp'dq' 
%((\widehat\Phi^t_0(q)-q')(p')-(\widehat\Phi^t_0(p)-p')\widehat\Phi^t_0(q))
\ff{(q',p')}\cdot % (q'-p')\cdot
(\widehat\Phi^t_0(q,p)-(q',p'))
\widehat F(\widehat\Phi^t_0(q,p)-(q',p'))\widehat{G}(q',p')e^{-ikt}\nonumber\\
=
\frac1{(2\pi)^d}\int dtdp'dq' 
%((\widehat\Phi^t_0(q-q'))\widehat\Phi^t_0(p')-(\widehat\Phi^t_0(p-p'))\widehat\Phi^t_0(q))
\widehat\Phi^t_0(q',p')\cdot\widehat\Phi^t_0(q-q',p-p')
\widehat F(\widehat\Phi^t_0(q-q',p-p'))\widehat{G}(\widehat\Phi^t_0(q',p'))e^{-ikt}\nonumber\\
=
\frac1{(2\pi)^d}\int dtdp'dq' 
((q-q')p'-(p-p')q')
\widehat F(\widehat\Phi^t_0(q-q',p-p'))\widehat{G}(\widehat\Phi^t_0(q',p'))e^{-ikt}\nonumber\\
=
\frac1{(2\pi)^d}\int dtdp'dq'dt'\sum_{k'\in\Z^d}e^{-ik'(t'-t)}
((q-q')p'-(p-p')q')
\widehat F(\widehat\Phi^t_0(q-q',p-p'))\widehat{G}(\widehat\Phi^{t'}_0(q',p'))e^{-ikt}\nonumber\\
=
\sum_{k'\in\Z^d}\int dp'dq'
((q-q')p'-(p-p')q')
\widehat {F}_{(k-k')}(q-q',p-p')\widehat {G}_{(k')}(q',p')\nonumber\\
=
\sum_{k'\in\Z^d}\int dp'dq'
((q-q')p'-(p-p')q')
\widehat {F_{(k-k')}}(q-q',p-p')\widehat
  {G_{(k')}}\fff{(q',p').}\nonumber % (q',p')\nonumber
%\\
%=
%\sum_{k'\in\Z^d}\int dp'dq'
%((q-q')p'-(p-p')q')
%\widehat {F_{(k-k')}}(q-q',p-p')\widehat {G_{(k')}}(q',p').\nonumber
\end{eqnarray}

In the case $\cP=T^*\T^d$, $\widehat{\{F,G\}_{(k)}}(q,p)=\widehat{\{F,G\}}(q,p)\delta_{k,p}$, so
 \begin{eqnarray}
\widehat{{\{{F,G}\}}_{(k)}}(q,p)\nonumber\\
=
\delta_{k,p}\int ((q-q')p'-(p-p')q')\widehat F(q-q',p-p')\widehat{G}(q',p')dq'dp'\nonumber\\
=
%\sum_{k'\in\Z^d}\delta_{k-k',p-p'}\delta_{k',p'}\int 
\bcf \sum_{k'\in\Z^d} \int \delta_{k-k',p-p'}\delta_{k',p'} \ecf
((q-q')p'-(p-p')q')\widehat F(q-q',p-p')\widehat{G}(q',p')dq'dp'\nonumber\\
=
\sum_{k'\in\Z^d}\int dp'dq'
((q-q')p'-(p-p')q')
\widehat {F_{(k-k')}}(q-q',p-p')
\widehat{G_{(k')}}(q',p').\nonumber
\end{eqnarray}
\end{proof}
Using now $|q|\leq|q-q'|+|q'|,|p|\leq|p-p'|+|p'|, |k|\leq|k-k'|+|k'|$ and 
%$xe^{\rho' x}\leq \frac 1{e(\rho''-\rho')}e^{\rho''x}.\ \forall x\geq 0$, one get, 
\bcf $xe^{\rho' x}\leq \frac 1{e(\rho''-\rho')}e^{\rho''x}$ for all $x\geq 0$, one gets,\ecf
by \eqref{sin},

\begin{eqnarray}
\norm{\{F,G\}}_{\rho'}=\sum_{k\in\Z^d}\int dqdp
|\widehat{\{F,G\}_{(k)}}(q,p)|e^{\rho'(|q|+|p|+|k|)}\ \ \ \ \ \label{esti}\\
\leq
\sum_{k,k'}\int dqdq'dpdp'
(|q-q'||p'|+|p-p'||q'|)
|\widehat {F_{(k-k')}}(q-q',p-p')\widehat {G_{(k')}}(q',p')|e^{\rho'(|q|+|p|+|k|)}\nonumber\\
\leq 
\sum_{k,k'}\int
(|q-q'||p'|+|p-p'||q'|)
|\widehat {F_{(k-k')}}(q-q',p-p')\widehat {G_{(k')}}(q',p')|\nonumber\\
e^{\rho'(|q-q'|+|p-p'|+|k-k'|+|q'|+|p'|+|k'|)}
 dqdq'dpdp'\nonumber\\
 \leq 
\sum_{k,k'}\int
(|q-q'|+|p-p'|)
|\widehat {F_{(k-k')}}(q-q',p-p')\widehat {G_{(k')}}(q',p')|\nonumber\\
\frac1{e(\rho''-\rho')}e^{\rho'(|q-q'|+|p-p'|+|k-k'|)+\rho''(|q'|+|p'|+|k'|)}
 dqdq'dpdp'\nonumber\\
 \leq 
\sum_{k,k'}\int
(|q|+|p|)
|\widehat {F_{(k)}}(q,p)\widehat {G_{(k')}}(q',p')|\nonumber\\
\frac1{e(\rho''-\rho')}e^{\rho'(|q|+|p|+|k'|)+\rho''(|q'|+|p'|+|k'|)}
 dqdq'dpdp'\nonumber\\
 \leq 
\sum_{k,k'}\int
%(|q|+|p|)
|\widehat {F_{(k)}}(q,p)\widehat {G_{(k')}}(q',p')|\nonumber\\
\frac1{e(\rho''-\rho')}\frac1{e(\rho-\rho')}e^{\rho(|q|+|p|+|k'|)+\rho''(|q'|+|p'|+|k'|)}
 dqdq'dpdp'\nonumber
\end{eqnarray}
since $\rho'<\rho''\leq\rho$ and one easily concludes.
\end{proof}
\ec
The same argument, used this time the weighted sum in $k$, leads to the next result.
\begin{lemma}\label{peutetre}
%[\cite{BGP,GP,PS}]
\[\norm{\{X_0,G\}}_{\rho'}\leq \frac {1}{e(\rho-\rho')}\norm{G}_\rho.\]
\end{lemma}
\begin{proof}
We first remark that
\begin{eqnarray}
  \{X_0,G\}_{(k)} = \int \{X_0,G\}\circ\Phi^t_0e^{ikt}dt 
  = \int \{X_0,G\circ\Phi^t_0\}e^{ikt}dt
= \{X_0,G_{(k)}\} = ik\cdot\omega G_{(k)}\nonumber
\end{eqnarray}
by~\eqref{eq:Gkhomog}, and we easily concludes using again $|k|e^{\rho'|k|}\leq\frac {1}{e(\rho-\rho')}e^{\rho|k|}$.
\end{proof}

\section{Estimating commutators}   \label{appCommut}
It has been proven in \cite{BGP} for $L^2(\R^d)$ and \cite{GP,PS} for $L^2(\T^d)$ the following Lemma.
\begin{lemma}\label{tout}
\bl{Suppose $0<\rho'<\rho$ and $F,G\in J_\rho$. Then}
\[
\norm{\frac1{i \hbar}[F,G]}_{\rho'}\leq \frac{\bcg 1\ec}{e^2(\rho-\rho')(\rho''-\rho')}\norm{F}_\rho\norm{G}_{\rho''}
\]
\[
\norm{\frac1{i\hbar}[X_0,G]}_{\rho'}\leq \frac{1}{e(\rho-\rho')}\norm{G}_\rho
\]
\end{lemma}
\begin{proof}
  Since the evolution by $X_0$ commutes with quantization by Corollary
  \ref{homweyl} (see Appendix~\ref{allthat}), we get that
  $\sig_{[F,G]_k/i\hbar}=(\sig_{[F,G]/i\hbar})_k,\ \forall k\in\Z^d$.
Moreover, by Lemma \ref{poiss}, $\sig_{[F,G]/i\hbar}=A(\sig_F\otimes\sig_G)$. Therefore
\be\label{compfou}
\widehat\sig_{[F,G]/i\hbar}(p,q)=\int \frac1\hbar\sin{{\hbar}((q-q')p'-(p-p')q')}\widehat{\sig_ F}(q-q',p-p')\widehat\sig_{G}(q',p')dq'dp'
\ee
a formula similar to \eqref{class} by the change $((q-q')p'-(p-p')q')\to\frac1\hbar\sin{{\hbar}((q-q')p'-(p-p')q')}$. The proof of the first inequality is identical to the one of Lemma \ref{poisson} modulo this change up to \eqref{sin}, and the rest of the proof, after \eqref{esti} is vertabim the same using the inequality
\[
|\frac1\hbar\sin{{\hbar}((q-q')p'-(p-p')q')}|\leq|((q-q')p'-(p-p')q')|.
\]
The proof of the second inequality is similar to the proof of Lemma \ref{peutetre}.
\end{proof}

\end{appendix}
\vskip 1cm
\textbf{Acknowledgments}: This work has been partially carried out thanks to the support of the A*MIDEX project (n$^o$ ANR-11-IDEX-0001-02) funded by the ``Investissements d'Avenir" French Government program, managed by the French 
National Research Agency (ANR). T.P. thanks also the Dipartimento di
Matematica, Sapienza Universit\`a di Roma, for its kind hospitality
during the completion of this work.
\bl{D.S. thanks Fibonacci Laboratory (CNRS UMI~3483), the Centro Di
Ricerca Matematica Ennio De Giorgi and the Scuola Normale Superiore di
Pisa for their kind hospitality. D.S.'s work has received funding from the
French National Research Agency under the reference ANR-12-BS01-0017.}

The two authors would like to thank warmly the referee for a careful reading of the manuscript and very interesting suggestions concerning further directions of research exposed in Remark \ref{referee} above.

%%%%%%%%%%%%%%%%%%%%%%%%%%%%%%%%%%%%%%%%%%%%%%%%%%%%%%%%%%%%%%%%%
%%%%%%%%%%%%%%%%%%%%%%%%%%%%%%%%%%%%%%%%%%%%%%%%%%%%%%%%%%%%%%%%%

%\newpage

%\begin{flushleft}
%
%\end{flushleft}

\end{document}